\documentclass[11pt]{article}
\usepackage{times,amsmath,amsthm,amssymb,graphicx,hyperref,float,xcolor}
\usepackage{fancyhdr}     
\usepackage{setspace}    
\usepackage{booktabs}    
\usepackage{longtable}   
\usepackage{multirow}
\usepackage{textcomp}    
\usepackage{bm}         
\usepackage{pdfpages}    
\newtheorem{defn}{Definition}
\newtheorem{prpstn}{Proposition}
\newtheorem{nt}{Note}
\newtheorem{rmrk}{Remark}

\newtheorem{lmm}{Lemma}
\newtheorem{thrm}{Theorem}
\newtheorem{xmpl}{Example}

\newcommand{\R}{\mathbb{R}}
\newcommand{\C}{\mathbb{C}}

\newcommand{\ds}{\displaystyle}

\newcommand{\bx}{\mathbf{x}}

\newcommand{\bbeta}{\bm{\beta}}
\newcommand{\bxi}{\bm{\xi}}
\begin{document}
\global\def\refname{{\normalsize \it References:}}
\baselineskip 12.5pt
%
%
%
\title{\LARGE \bf Regularization techniques for estimating the source in a complete parabolic equation in $\R^n$}

\date{}

\author{\hspace*{-10pt}
\begin{minipage}[t]{2.3in} \normalsize \baselineskip 12.5pt
\centerline{GUILLERMO FEDERICO UMBRICHT}
\centerline{Departamento de Matem\'atica, Facultad de Ciencias Empresariales, Universidad Austral}
\centerline{Paraguay 1950, Rosario, Santa Fe, ARGENTINA}
\centerline{Consejo Nacional de Investigaciones Cient{\'i}ficas y T\'ecnicas (CONICET)}
\centerline{Godoy Cruz 2290, CABA, ARGENTINA}
\vspace{0.5cm}
\centerline{DIANA RUBIO}
\centerline{ITECA (UNSAM-CONICET), CEDEMA, ECyT, Universidad
Nacional de General San Mart\'in}
\centerline{25 de mayo y Francia, San Mart\'in, Buenos Aires, ARGENTINA}
\end{minipage}
%
%
\\ \\ \hspace*{-10pt}
\begin{minipage}[b]{6.9in} \normalsize
\baselineskip 12.5pt {\it Abstract:}
In this article, the problem of identifying the source term in transport processes given by a complete parabolic equation is studied mathematically from noisy measurements taken at an arbitrary fixed time. The problem is solved analytically with Fourier techniques and it is shown that this solution is not stable. 
Three single parameter families of regularization operators are proposed to dealt with the instability of the solution.
Each of them is designed to compensate the factor that causes the instability of the inverse operator.  Moreover, a rule of  choice for the regularization parameter is included and a  H\"older error bound type is
obtained for each estimation. Numerical examples of different characteristics are presented to demonstrate the benefits of the proposed strategies. 
\\ [4mm] {\it Key--Words:}
Transfer problems, Parabolic equation, Inverse problems, Ill-posed problems, Regularization operators.
\end{minipage}
\vspace{-10pt}}

\maketitle

\thispagestyle{empty} \pagestyle{empty}
%
%
\section{Introduction}
\label{s:1} \vspace{-4pt}

Complete parabolic equations are used to model transport phenomena. These are highly studied due to the number of direct applications they are related to in different fields of physics, as they are used to analyze the transfer of mass, energy, and information in different processes \cite{Bird02, Hangos01}. Among the most important applications, we can mention the heat transfer of a body immersed in a moving fluid \cite{Umbricht21d}, heat transfer in biological tissue \cite{Pen48}, and the concentration of a certain compound in a flow of water \cite{Basha93}.

On the other hand, the problem of source determination has been analyzed in recent years in various areas of applied physics and has garnered considerable attention in current research. It finds applications in fields such as heat conduction \cite{Alifa75, Elden00, Umbricht19d, Zhao11}, crack identification \cite{Zeng96}, electromagnetic theory \cite{Banks15, Ciaret10, Elbadia00}, geophysical prospecting \cite{Beroza88}, contaminant detection \cite{Li06}, and detection of tumor cells \cite{Macleod99}, among others.

Several methods are available in the literature for determining a source, and among the most important tools are the potential logarithmic method \cite{Ohe94}, the projective method \cite{Nara03}, Green functions \cite{Hon10}, limit element methods of dual reciprocity \cite{Farcas03}, the method of dual reciprocity \cite{Sun97}, and the method of fundamental solution \cite{Jin07}.

In the context of various transport processes or complete parabolic differential equations, there is a limited number of articles published for the general case \cite{Umbricht19a, Umbricht19b, Umbricht21}, with the majority of available literature focusing on the one-dimensional heat equation. Different methods, techniques, and strategies have been employed to retrieve sources in this equation, as seen in \cite{Farcas06, Ahmadabadi09, Johansson07}. Many articles analyze specific cases with simplifications or restrictions in mathematical models, source types, boundary conditions, or selected domains, such as \cite{Farcas06, Ahmadabadi09, Johansson07, Liu09, Savateev95, Cannon98, Yan08}. The most commonly used methods in these cases include the limit element method \cite{Farcas06, Johansson07}, the fundamental solution method \cite{Ahmadabadi09, Yan08}, the Ritz-Galerkin method \cite{Ras11}, the finite difference method \cite{Yan10}, the meshless method \cite{Yan09}, the conditional stability method \cite{Yama93}, and the shooting method \cite{Liu09}.

The problem of source determination is considered ill-posed in the sense of Hadamard \cite{Hada24}, as the solution does not depend continuously on the data. Regularization methods \cite{Engl96, Kirsch11} play a crucial role in estimating unstable solutions. Among the most commonly used are the iterative regularization method \cite{Honchbruck09, Johansson08}, the simplified Tikhonov regularization method \cite{Fu04, Cheng07, Cheng08, Natterer84, Yang10b}, the modified regularization method \cite{Umbricht19d, Umbricht21, Yang10b, Yang10, Yang11, Zhao14}, Fourier truncation \cite{Yang11}, and the mollification method \cite{Yang14}.

When it comes to source determination in parabolic equations, \cite{Sivergina2003} focuses on a convection-diffusion equation, while \cite{Elden00, Yan10, Yang10, Zhao14, Dou09, Dou09b, Martin96, Trong05, Trong06} exclusively consider diffusion. More recently, the problem of finding the source term in a complete parabolic equation has been addressed using the quasi-reversibility method, as seen in \cite{Li19}. Notably, this method extends its applicability to solving the inverse source problem in nonlinear parabolic equations \cite{Le20}.

The issue addressed in this article stems from modeling various transportation problems and involves identifying the source term, which is independent of time, in a complete parabolic evolutionary equation using measurements or noisy data taken at an arbitrary fixed time. This problem is ill-posed, as high-frequency components of arbitrarily small data errors can result in disproportionately large errors in the solution. To tackle this challenge, three uniparametric families of regularization operators are devised to counteract the instability of the inverse operator. These regularization operators create families of well-posed problems that approximate the originally ill-posed problem. Additionally, a guideline is provided for selecting the regularization parameter. The stability and convergence of each method are analyzed, and an optimal H"older bound is derived for the error of each estimate.

The three strategies presented here generalize the ideas used by other authors for the case of the one-dimensional heat equation. Here, they are applied to a complete parabolic equation where temperature measurements can be taken at any instant. On the other hand, the proposal is framed within the theory of operators, giving rise to applications in more general contexts and problems.
Any important application is the detection of contaminants in groundwater layers, this is a problem that concerns all urbanized cities. Being able to determine the focus of contamination from measurements in a particular place, minimizes the costs used for the search \cite{Li06,Das17b,Sun96}.
Another application is the estimation of the metabolic heat in a biological tissue \cite{Umbricht19b} using the Pennes model \cite{Pen48}.
Any existing abnormality within the body can lead to variations in temperature and heat flux at the surface. The presence of a tumor produces local inflammation, increased metabolic activity,
  among other symptoms. Due to this, the diseased cells act as a source of temperature that produces abnormal thermal profiles in the skin and its measurements can be used to identify, locate and characterize the sick cells.
   The two problems mentioned above can be addressed with the tools that are explained in this paper.
There are many other applications in different disciplines, these two are mentioned as an example to highlight the importance of the problem and the proposed methods.

To illustrate the performance of the proposed regularizations and in order to compare the different regularization operators introduced in this article; numerical examples of different characteristics are included.
\section{Source identification}

In this section, the inverse problem to be studied is formally presented, an analytical expression is given for the solution of the problem of interest, and it is shown that the problem is ill-posed.

\subsection{Presentation of the problem} 

Given $u:\R^n \times \R^+\longrightarrow \R$, we seek to determine the source term in the following complete parabolic equation in an unbounded domain

\begin{equation}
\dfrac{\partial u}{\partial t}({\bx},t)=\alpha^2 \Delta u(\bx,t)-\bbeta \cdot \nabla (u(\bx,t))-\nu u(\bx,t)+f(\bx), \qquad  \bx \in \R^{n}, \,\,\, t>0,
\label{Ec_1_Mod}
\end{equation}
where $\Delta$,$\nabla$ denote the Laplacian and Nabla operators, respectively, and $``\cdot"$ represents the usual inner product in $\R^{n}$.
Without loss of generality, the null initial condition is considered, i.e.,

\begin{equation}
\label{C_I}
u(\bx,t)=0, \qquad  \bx \in \R^{n}, \,\,\, t=0.  
\end{equation}

The determination of the source term in Equation \eqref{Ec_1_Mod} with the condition \eqref{C_I}, is carried out using experimental or simulated noisy data, at an instant of time $t_0$. 

\begin{equation}
\label{Med_Ruidosas}
u(\bx,t)=y_{\delta}(\bx), \qquad  \bx \in \R^{n}, \,\,\, t=t_0>0,  
\end{equation}
where $y_{\delta}$ represents the noisy data or measurements and $\delta$ is the noise in the data. It is also assumed that said noise is bounded, that is,

\begin{equation}
\label{noiselevel}
    ||y(\bx)-y_{\delta }(\bx)||_{L^2(\R^{n})}\le \delta, \qquad 0<\delta\le \delta_M,\\
\end{equation} 
where $\delta_M \in \R^{+} $ is the maximum noise level tolerated. In practice $\delta_M$ is obtained from measurement, instrumentation and calibration errors.

\subsection{Problem solution}

The source estimation problem will be solved using the $n$-dimensional Fourier transform. It is included here for completeness reasons.

\begin{defn}\label{DefTransf} 

 Let $g \in L^2(\R^n)$, the $n$-dimensional Fourier transform \cite{Marks09} is defined by
\begin{equation*}
\widehat{g}(\bxi) :=\left(\dfrac{1}{\sqrt{2\pi}}\right)^n \ds\int\limits_{\R^n}{e^{-i \bxi \cdot \bx} \, g (\bx) \,  d\bx}, \qquad   \bxi \in \R^n.
\end{equation*}

Let $\widehat{g} \in L^2(\R^n)$, the Fourier antitransform is defined by
\begin{equation*}
\label{defantitransfn}
g(\bx) :=\left(\dfrac{1}{\sqrt{2\pi}}\right)^n \ds\int\limits_{\R^n}{e^{i \bxi \cdot \bx} \, \widehat{g} (\bxi) \,  d\bxi}, \qquad   \bx \in \R^n.
\end{equation*}
\end{defn}

\begin{prpstn} \label{Prop_gradiente_y_Laplaciano}
The $n$-dimensional Fourier transform is a linear integral operator that satisfies the following property.
Let $g \in L^2(\R^n)$ hence:

\begin{equation*} 
\widehat{\nabla g}(\bxi)= i \, \bxi \, \widehat{g}(\bxi), \qquad \widehat{\Delta g}(\bxi)=- \left\|\bxi\right\|^2 \, \widehat{g}(\bxi), \qquad \bxi \in \R^n.
\end{equation*} 
\end{prpstn}

From the use of the definition \ref{DefTransf} and the proposition \label{PropNablayLap} it is possible to find the analytical solution of the problem of interest given by the equations \eqref{Ec_1_Mod}-\eqref{noiselevel}. This is given in the following result.

\begin{thrm}[Solving the source identification problem] \label{teorema_Fuente_Inversa}
For $\bbeta \in \R^{n}$ and $\alpha^2,\nu, t_0, \delta, \delta_M \in \R^+$ such that $\delta<\delta_M $, are the functions $u(\cdot, t), f(\cdot), y(\cdot), y_\delta(\cdot) \in L^2(\R^n)$ with $||y-y_{\delta}||_{L^2(\R^n)}\le \delta$ satisfying the parabolic problem

\begin{equation}
\label{transpeqn}
\begin{cases} 
\dfrac{\partial u}{\partial t}({\bx},t)=\alpha^2 \Delta u(\bx,t)-\bbeta \cdot \nabla (u(\bx,t))-\nu u(\bx,t)+f(\bx), \quad & \bx \in \R^{n}, \,\,\, t>0, \\
u(\bx,t)=0, \quad  & \bx \in \R^{n}, \,\,\, t=0, \\
u(\bx,t)=y_{\delta}(\bx), \quad & \bx \in \R^n, \,\,\, t= t_0>0. 
\end{cases}
\end{equation}

Then, the expression for the source is given by

\begin{equation*}
f_{\delta}(\bx) =\left(\dfrac{1}{\sqrt{2\pi}}\right)^n \ds\int\limits_{\R^n} e^{i \bxi \cdot \bx} \Lambda (\bxi) \left[\left(\dfrac{1}{\sqrt{2\pi}}\right)^n \ds\int\limits_{\R^n} e^{-i \bxi \cdot \bx} y_{\delta}(\bx)d\bx\right]d\bxi,
\end{equation*}
where

\begin{equation*}
\Lambda (\bxi )=\dfrac{z(\bxi)}{1-e^{-z(\bxi)\, t_0}},
\end{equation*}
with

\begin{equation*}
z(\bxi)=\alpha^2 \left\|\bxi\right\|^2 + i\,\bbeta \cdot \bxi + \nu \, \in \C.
\end{equation*}
\end{thrm}

\begin{proof}

For the proof of theorem \ref{teorema_Fuente_Inversa} it is used the $n$-dimensional Fourier transform given in \ref{DefTransf} on the space variables of the system \eqref{transpeqn}. Then the proposition \ref{Prop_gradiente_y_Laplaciano} is used to obtain,

\begin{equation}
\label{transpeqn2}
\begin{cases} 
\widehat{\dfrac{\partial u}{\partial t}}({\bxi},t)=-(\alpha^2 \left\|\bxi\right\|^2 +i \bbeta \cdot \bxi + \nu) \, \widehat{u}(\bxi,t) + \widehat{f}(\bxi), \qquad & \bxi \in \R^{n}, \,\,\, t>0, \\
\widehat{u}(\bxi,t)=0, \qquad  & \bxi \in \R^{n}, \,\,\, t=0, \\
\widehat{u}(\bxi,t)=\widehat{y_{\delta}}(\bxi), \qquad & \bxi \in \R^n, \,\,\, t= t_0>0, 
\end{cases}
\end{equation}
where $\bxi$ is the $n$-dimensional Fourier variable. Equivalently, the identification problem \eqref{transpeqn} can be rewritten fron \eqref{transpeqn2} in frequency space as follows

\begin{equation}
\label{transpeqn3}
\begin{cases} 
\widehat{\dfrac{\partial u}{\partial t}}({\bxi},t)=-z(\bxi) \, \widehat{u}(\bxi,t) + \widehat{f}(\bxi), \qquad & \bxi \in \R^{n}, \,\,\, t>0, \\
\widehat{u}(\bxi,t)=0, \qquad  & \bxi \in \R^{n}, \,\,\, t=0, \\
\widehat{u}(\bxi,t)=\widehat{y_{\delta}}(\bxi), \qquad & \bxi \in \R^n, \,\,\, t= t_0>0, 
\end{cases}
\end{equation}
where $z(\bxi)$ is given by the expression,

\begin{equation}
\label{z}
z(\bxi)=\alpha^2 \left\|\bxi\right\|^2 + i\,\bbeta \cdot \bxi + \nu.
\end{equation}

The system \eqref{transpeqn3} is represented by a first order non-homogeneous  differential equation with initial condition. The analytical solution to this equation is, 

\begin{equation}
\label{solutionu}
\widehat{u}(\bxi,t) = \dfrac{1-e^{-z(\bxi)\, t}}{z(\bxi)}\widehat{f}(\bxi).
\end{equation}

Because $\widehat{u}(\bxi,t_0)=\widehat{y_{\delta}}(\bxi)$, evaluating Equation \eqref{solutionu} at $t=t_0$ produces a linear expression for the source in frequency space, 

\begin{equation}
\label{illf}
\widehat{f_{\delta}}(\bxi )=\Lambda (\bxi )\widehat{y_{\delta}}(\bxi ),
\end{equation}
where

\begin{equation}
\label{Lambda}
\Lambda (\bxi )=\dfrac{z(\bxi)}{1-e^{-z(\bxi)\, t_0}}.
\end{equation}

Equivalently, using the definition \ref{DefTransf} (of the Fourier antitransform) results from the Equation\eqref{illf}
 
\begin{equation}
\label{solucion_fuente}
f_{\delta}(\bx) =\left(\dfrac{1}{\sqrt{2\pi}}\right)^n \ds\int\limits_{\R^n} e^{i \bxi \cdot \bx} \Lambda (\bxi) \left[\left(\dfrac{1}{\sqrt{2\pi}}\right)^n \ds\int\limits_{\R^n} e^{-i \bxi \cdot \bx} y_{\delta}(\bx)d\bx\right]d\bxi,
\end{equation}

which ends the proof.
\end{proof}

\subsection{Ill-posed problem}

The problem of identifying the source in a complete parabolic equation from noisy measurements turns out to be a ill-posed problem in the sense of Hadamard \cite{Hada24} since the solution does not depend continuously on the data. This fact can be seen in the following result.

\begin{thrm} [The problem is ill-posed]\label{Teo_Prob_mal_Planteado}

Under the assumptions used so far, the identification problem \eqref{transpeqn} given in Theorem \ref{teorema_Fuente_Inversa}. It is an ill-posed problem in Hadamard's sense, since the solutions do not vary continuously with the data.
\end{thrm}

\begin{proof}

We denote $\widehat{f_{\delta}}(\bxi)=\Lambda (\bxi )\widehat{y_{\delta}}(\bxi )$. It's easy to see that

\begin{equation}
\label{error_f_fdelta}
\|\widehat{f} - \widehat{f_{\delta}}\|_{L^2(\R^n)} = \| \Lambda(\bxi) ( \widehat{y}(\bxi )- \widehat{y_{\delta}}(\bxi ))\|_{L^2(\R^n)}=
\left\| \Lambda(\bxi)(\widehat{y}(\bxi )- \widehat{y_{\delta}}(\bxi ))\right\|_{L^2(\R^n)},
\end{equation}
on the other hand, using the Equations \eqref{z} and \eqref{Lambda} it is found that,

\begin{eqnarray}
 \label{lim}
| \Lambda (\bxi )|=\left| \dfrac{  z(\bxi)}{1-e^{-z(\bxi)\, t_0}}\right|
=\left| \dfrac{\alpha^2 \left\|\bxi\right\|^2 + i\,\bbeta \cdot \bxi + \nu}{1-e^{-(\alpha^2 \left\|\bxi\right\|^2 + i\,\bbeta \cdot \bxi + \nu)\, t_0}}\right| \geq 
\dfrac{  |\alpha^2 \left\|\bxi\right\|^2 + i\,\bbeta \cdot \bxi + \nu|}{1+e^{-(\alpha^2 \left\|\bxi\right\|^2 + \nu)\, t_0}} 
,
\end{eqnarray}
from \eqref{lim} it is evident that $\Lambda (\bxi)$ is not bounded, since it tends to infinity as $\left\|\bxi\right\| \to \infty$. As can be seen in
\eqref{error_f_fdelta} this fact amplifies the error of the measurements\index{Measurements} at high frequencies and this can lead to a large estimation error 
$ \|\widehat{f} - \widehat{f_{\delta}}\|_{L^2(\R^n)}$ even for very small observation or measurement errors. In other words, the solution to the identification problem \eqref{transpeqn} does not vary continuously with the data (see \cite{Engl96}).

\end{proof}

\section{Regularization operators}\label{s:3}\vspace{-4pt}

When an inverse problem is ill-posed, a regularization method is usually applied to stabilize the solution. In this section, three regularization operators are proposed for comparative purposes. The existence of the regularization parameter leading to three convergent methods is proved, and basic theoretical issues related to regularization operators are included.
  Readers unfamiliar with this topic may find more information at\cite{Engl96, Kirsch11}.

\subsection{Regularization solutions}

To stabilize the ill-posed problem, regularization operators will be used.

\begin{defn}
Let $\mathbb{X}$ and $\mathbb{Y}$ be Hilbert spaces and $T : \mathbb{Y} \longrightarrow \mathbb{X}$ a linear bounded operator. A regularization strategy for $T$ is a family of linear bounded operators satisfying
\begin{equation*}
\label{DefinitionR}
R_{\mu} : \mathbb{Y} \longrightarrow \mathbb{X}, \quad \mu>0, \quad / \lim_{\mu \to 0^+} R_{\mu} y = Ty, \quad \forall y \in \mathbb{Y}.
\end{equation*}
\end{defn}
For our particular case, the parametric families of linear operators are defined by
$R_\mu ^i: L^2(\R^n) \to L^2(\R^n)$ with $\mu \in \R^{+}$ and $i=1,2,3$; such that
\begin{equation}
\label{familyR}
R_{\mu}^{1} \, \widehat{y} := \dfrac{\Lambda (\bxi)}{1+\mu^2 \left\|\bxi\right\|^2} \, \widehat{y} , \qquad
R_{\mu}^{2} \, \widehat{y} := \dfrac{\Lambda (\bxi)}{1+\mu^2 \left\|\bxi\right\|^4} \, \widehat{y} , \qquad
R_{\mu}^{3} \, \widehat{y} := \dfrac{\Lambda (\bxi)}{e^{\mu^2 \,\left\|\bxi\right\|^ 2/4}} \, \widehat{y} ,
\end{equation}
where $\Lambda (\bxi )$ is defined in \eqref{Lambda}, $R_{\mu}^{i}$ with $i=1,2,3$, are regularization strategies for $\Lambda (\bxi )$ and $\mu$ is the regularization parameter.

\begin{nt}
Note that the denominators of the expressions \eqref{familyR} that define the linear operators $R_{\mu}^{i}$ with $i=1,2,3$, were introduced in the solution only for stabilization purposes.
\end{nt}

\begin{thrm}  [Convergent regularization operators] 
Consider the source identification problem \eqref{transpeqn}.
%
%
Let $u(\cdot,t), f(\cdot) \in L^2(\R^n) $ that satisfy the following differential equation with initial condition
\begin{equation}
\label{illpp}
\begin{cases}
\dfrac{\partial u}{\partial t}({\bx},t)=\alpha^2 \Delta u(\bx,t)-\bbeta \cdot \nabla (u(\bx,t))-\nu u(\bx,t)+ f(\bx), \quad & \bx \in \R^{n}, \,\,\, t>0, \\
u(\bx,t)=0, \quad & \bx \in \R^{n}, \,\,\, t=0 \\
\end{cases}
\end{equation}
and let $\{R_{\mu}^{i}\}$ with $i=1,2,3$ be the families of operators defined in \eqref{familyR}.
Then, for all $y(\bx)=u(\bx,t_0)$ there exists an a priori parameter choice rule for $\mu>0$ such that the pairs $(R_{\mu}^ {i},\mu)$ for $i=1,2,3$
are convergent regularization methods for the identification problem \eqref{illpp}.
\end{thrm}
%
\begin{proof}
The operators $\{R_{\mu}^{i}\}$ with $i=1,2,3$ defined in \eqref{familyR} are continuous for all $\bxi \in \R^n$ and are bounded, since it is evident that,
\begin{equation*}
\begin{split}
0 \leq\lim_{\left\|\bxi\right\| \to \infty} R_{\mu}^{i}<\infty, \qquad i=1,2,3.
\end{split}
\end{equation*}
Therefore, for each parameter $\mu>0$, the operators $R_{\mu}^{i}$ with $i=1,2,3$ are linear, continuous and it has to be
\begin{equation*}
\lim _{\mu \to 0^+} R_{\mu}^{i} \, \widehat{y} = \Lambda \, \widehat{y}, \qquad  i=1,2,3,
\end{equation*}
where $ \widehat{y} \in L^{2}(\R^n)$, then $R_{\mu}^{i}$ with $i=1,2,3$ are regularization strategies for $ \Lambda$. Therefore, according to \cite{Engl96},
there exist \emph{a priori} parameter choice rules $\mu$ such that $(R_{\mu}^{i}, \mu)$ with $i=1,2,3$ are convergent regularization methods to solve \eqref {illf}.
\end{proof}

The regularized solution of the inverse source identification problem, in frequency space, is given by
\begin{equation}
\label{ffregtransf}
\widehat{f}_{\delta ,\mu}^{i} (\bxi) = R_{\mu}^{i}  \, \widehat{y}_\delta (\bxi), \qquad i=1,2,3.
\end{equation}
Therefore, from the equation \eqref{ffregtransf} an approximate expression for the function $f$ is obtained, solution of the problem given in \eqref{illpp}. This is,
\begin{equation*}
\label{ffreg}
f_{\delta ,\mu}^{i}(\bx) =\left(\dfrac{1}{\sqrt{2\pi}}\right)^n \ds\int\limits_{\R^n } e^{i \bxi \cdot \bx} R_{\mu}^{i} \, \widehat{y}_{\delta} (\bxi) \, d\bxi , \qquad i=1,2 ,3,
\end{equation*}
equivalently
\begin{equation}
\label{ffreg2}
f_{\delta ,\mu}^{i}(\bx) =\left(\dfrac{1}{\sqrt{2\pi}}\right)^n \ds\int\limits_{\R^n } e^{i \bxi \cdot \bx} R_{\mu}^{i} \, \left[\left(\dfrac{1}{\sqrt{2\pi}}\right)^n \ds \int\limits_{\R^n} e^{-i \bxi \cdot \bx} \, y_{\delta}(\bx)d\bx\right] \, d\bxi
, \qquad i=1,2,3.
\end{equation}
Using in the equation \eqref{ffreg2}, the definitions of the regularization operators $(R_{\mu}^{i}, \mu)$ with $i=1,2,3$ given by the equations \eqref{familyR}, three analytic expressions for the source estimate are obtained (one for each regularization strategy). These are,
\begin{equation*}
\label{f1}
f_{\delta ,\mu}^{1}(\bx) =\left(\dfrac{1}{\sqrt{2\pi}}\right)^n \ds\int\limits_{\R^n } e^{i \bxi \cdot \bx} \dfrac{\Lambda (\bxi)}{1+\mu^2 \left\|\bxi\right\|^2} \, \left[\left( \dfrac{1}{\sqrt{2\pi}}\right)^n \ds\int\limits_{\R^n} e^{-i \bxi \cdot \bx} \, y_{\delta} (\bx)d\bx\right] \, d\bxi,
\end{equation*}
\begin{equation*}
\label{f2}
f_{\delta ,\mu}^{2}(\bx) =\left(\dfrac{1}{\sqrt{2\pi}}\right)^n \ds\int\limits_{\R^n } e^{i \bxi \cdot \bx} \dfrac{\Lambda (\bxi)}{1+\mu^2 \left\|\bxi\right\|^4} \, \left[\left( \dfrac{1}{\sqrt{2\pi}}\right)^n \ds\int\limits_{\R^n} e^{-i \bxi \cdot \bx} \, y_{\delta} (\bx)d\bx\right] \, d\bxi
\end{equation*}
and
\begin{equation*}
\label{f3}
f_{\delta ,\mu}^{3}(\bx) =\left(\dfrac{1}{\sqrt{2\pi}}\right)^n \ds\int\limits_{\R^n } e^{i \bxi \cdot \bx} \dfrac{\Lambda (\bxi)}{e^{\mu^2 \,\left\|\bxi\right\|^2/4}} \, \left[\left(\dfrac{1}{\sqrt{2\pi}}\right)^n \ds\int\limits_{\R^n} e^{-i \bxi \cdot \bx } \, y_{\delta}(\bx)d\bx\right] \, d\bxi.
\end{equation*}

\subsection{Error Analysis}

In this section, a bound will be obtained for the error committed by each estimation.

\subsubsection{Auxiliary results}

In order to analyze the behavior of the proposed regularizations, some results are first introduced that will be used later to obtain an error bound between the source $f(\bx)$ and its estimates $f_{\delta ,\mu } ^{i}(\bx).$

\begin{nt} 
Some of the auxiliary results may seem trivial to the reader. However, and in order to carry out a complete and self-contained work, a brief demonstration is presented for each one of them.
\end{nt}

\begin{lmm}
\label{lemma1}
Let $\omega \in \C$ with $Re(\omega)>0$ then
$\ds \left\vert \dfrac{1}{1-e^{-\omega }} \right\vert \le \dfrac{1}{1-e^{-Re(\omega)}}.$
\end{lmm}
\begin{proof}

Euler's formula is used for complex numbers and it is found that,

\begin{equation*}
\ds{\left|1-e^{-\omega }\right|^2= \left(1-e^{-Re(\omega)}\right)^2+2e^{-Re(\omega)}\left(1-cos\left(Im(\omega)\right)\right)} \geq \left(1-e^{-Re(\omega)}\right)^2,
\end{equation*}
therefore, 
\begin{equation*}
\left|1-e^{-\omega }\right| \geq 1-e^{-Re(\omega)} \Longrightarrow \left\vert\dfrac{1}{1-e^{-\omega }} \right\vert \le \dfrac{1}{1- e^{-Re(\omega)}}.
\end{equation*}
\end{proof}

\begin{lmm}
\label{lemma2}
The function $f:\R^{+} \to \R$ defined by
$\ds f(x):=
\begin{cases}
   \dfrac{x}{1-e^{-x}}, & 0<x<1,\\
   \dfrac{1}{1-e^{-x}}, & 1\leq x,
   \end{cases}$
satisfies $f(x) < 2 $.
\end{lmm}
\begin{proof}
First, consider $f$ at $(0,1)$. Differentiating, in this case, we have
\begin{equation*}
f'(x)= \left(\dfrac{x}{1-e^{-x}}\right) '= \dfrac{1+e^{-x}(-1+x)}{(1 -e^{-x})^2} > 0,
\end{equation*}

the function $f$ is increasing in $(0,1)$. Then $\dfrac{x}{1-e^{-x}} \leq \dfrac{1}{1-e^{-1}}.$

On the other hand, for $x\geq1$, we have
\begin{equation*}
f'(x)=\left (\dfrac{1}{1-e^{-x}}\right) '=\dfrac{-e^{-x}}{(1-e^{-x} )^2} < 0,
\end{equation*}

the function $f$ is decreasing $ \forall x\geq1 $. Then $\dfrac{x}{1-e^{-x}} \leq \dfrac{1}{1-e^{-1}}.$

So, $f(x)\leq \dfrac{1}{1-e^{-1}} < 2$,
\end{proof}

\begin{lmm}
\label{lemma3}
Let $x \in \R-\left\{0\right\}$. If $0<\mu<1 $ we have $\dfrac{x^2}{(1-e^{-x^2})e^{(\mu^2 x^2)/4}} < \dfrac {4}{\mu^2}.$
\end{lmm}
\begin{proof}
Using the lemma \ref{lemma2} the proof is immediate since,
\begin{equation*}
  0 \leq x^2<1 \Longrightarrow \dfrac{x^2}{(1-e^{-x^2})e^{(\mu^2 x^2)/4}}<\dfrac{ 2}{e^{(\mu^2 x^2)/4}}<2<\frac{2}{\mu^2}<\frac{4}{\mu^2},
\end{equation*}
\begin{equation*}
  x^2\geq1 \Longrightarrow \dfrac{x^2}{(1-e^{-x^2})e^{(\mu^2 x^2)/4}}<\dfrac{2 x^ 2}{e^{(\mu^2 x^2)/4}}<\dfrac{2}{e^{\mu^2/4}}<\frac{4}{\mu^2},
\end{equation*}
\end{proof}

\begin{lmm}
\label{lemma4}
Let $x \in \R-\left\{0\right\}$. So $\dfrac{1-e^{-x^2}}{x^2}<1$.
\end{lmm}
\begin{proof}
Let $g:\R^{+} \to \R$ be defined by $g(x):=\dfrac{1-e^{-x^2}}{x^2}$. Differentiating we get,
\begin{equation*}
g'(x)= \dfrac{-2x\left[1-(x^2+1)e^{-x^2}\right]}{x^2}< 0, \quad \forall x>0 .
\end{equation*}
Then, the function $g$ is strictly decreasing\index{Function!decreasing} and also,
\begin{equation*}
\lim_{x \to 0^+}g(x)=1, \quad  \lim_{x \to +\infty}g(x)=0.
\end{equation*}
Then $g(x)<\lim_{x \to 0^+}g(x)=1, \quad \forall x>0.$
\end{proof}

\begin{lmm}
\label{lemma5}
Let $a,b,x \in \R^+$. So $\dfrac{x}{ax^2+b}\leq \dfrac{1}{2\sqrt{ab}}$.
\end{lmm}
\begin{proof}
Let $h:\R^{+} \to \R$ be defined by $h(x):=\dfrac{x}{ax^2+b}$. Differentiating we get,
\begin{equation*}
h'(x)= \dfrac{b-ax^2}{ax^2+b}, \quad \forall x>0.
\end{equation*}
Then the function $h$ has its only critical point at $x_0=\sqrt{\dfrac{b}{a}}$. Noting that $f(0)=0$
and $\lim_{x \to +\infty}h(x)=0$, it follows that $x_0$ is the absolute maximum of $h$.
So,
\begin{equation*}
h(x)=\dfrac{x}{ax^2+b}\leq h(x_0)=\dfrac{1}{2\sqrt{ab}}, \quad \forall x>0,
\end{equation*}
\end{proof}

\begin{lmm}
\label{lemma6}
Let $\rho \in \R$. If $0<\mu<1 $ we have $\dfrac{\vert \rho \vert}{1+\rho^2\mu^2} < \dfrac{1}{2 \mu^2}.$
Furthermore, for $\alpha^{2},\nu >0$ the following inequality is valid $\dfrac{\alpha^2\rho^2+\nu }{1+\rho^2 \mu^2} < \dfrac{1}{\mu^2} \left(\nu + \alpha^2\right).$
\end{lmm}
\begin{proof}
For all $\rho \in \R$ holds
\begin{equation*}
0 \leq \ds(1- \vert \rho \vert \mu )^2= 1- 2 \vert \rho \vert \mu + \vert \rho \vert ^2 \mu ^2 \Longrightarrow 1+ \vert \rho \vert ^2 \mu ^2 \geq 2 \vert \rho \vert \mu.
\end{equation*}
In addition, since $\mu>0$ the first inequality is obtained, $\dfrac{\vert \rho \vert}{1+\rho^2\mu^2} \le \dfrac{1}{2\mu} < \dfrac{1}{2\mu^2}.$

On the other hand, let $\ds k(\rho):=\frac{\alpha^2\rho^2+\nu
}{1+\rho^2\mu^2}$ then
$k'(\rho )=\dfrac{2\rho (\alpha^2-\nu \mu^2)}{(1+\rho^2\mu^2)^{2}}$ and $k $ has a single critical point $\rho =0$. Three cases are considered:
\begin{itemize}
\item $\alpha^2=\nu \mu^2$: we have $k(\rho )=\nu$, constant $\forall \rho \in \R$.
\item $\alpha^2<\nu \mu^2$: then the function $k$ reaches its global maximum value $\nu$ at $\rho=0$.
\item $\alpha^2>\nu \mu^2$: $k$ is an even and increasing function for $\rho>$0 with
  $\ds \lim_{\rho \to \pm \infty} k(\rho)=\dfrac{\alpha^{2}}{\mu^{2}}$ we have $k(\rho)\leq \ds \dfrac{\alpha^{2}}{\mu^{2}}$.
  \end{itemize}
Therefore,
\begin{equation*}
\dfrac{\alpha^2\rho^2+\nu }{1+\rho^{2}\mu^2}
  \le \max \left\{ \nu, \dfrac{\alpha^{2} }{\mu^{2} } \right\}\le \max \left\{ \dfrac{\nu}{\mu ^{2}}, \dfrac{\alpha^{2} }{\mu^{2} }\right\} < \dfrac{1}{\mu^2} \left(\nu + \alpha^2 \right).
\end{equation*}

\end{proof}

\begin{lmm}
\label{lemma7}
Let $\rho \in \R^+$. If $0<\mu<1 $ we have $\dfrac{\rho}{1+\rho^4\mu^2} < \dfrac{1}{\mu^2}.$
Also, for $\, \alpha^{2}, \nu >0$ the following inequality is valid $\dfrac{\alpha^2\rho^2+\nu }{1+\rho^{4}\mu ^2}   < \dfrac{1}{\mu^2} \left(\nu + \alpha^2\right).$
\end{lmm}
\begin{proof}
For all $\rho \geq 1$ the lemma \ref{lemma6} is used and the proof is immediate, since
\begin{equation*}
\dfrac{\rho}{1+\rho^4\mu^2} \le \dfrac{\rho }{1+\rho^2\mu^2} < \dfrac{1}{2 \, \mu ^2} < \dfrac{1}{ \mu^2}, \qquad \dfrac{\alpha^2\rho^2+\nu }{1+\rho^{4}\mu^2} \le \dfrac{\alpha^2\rho^2+\nu }{1+\rho^{2}\mu^2}
< \dfrac{1}{\mu^2} \left(\nu + \alpha^2\right).
\end{equation*}

On the other hand, for the case $0<\rho<1$, we have,
\begin{equation*}
\dfrac{\rho}{1+\rho^4\mu^2} < \dfrac{1}{1+\rho^4\mu^2} < 1 < \dfrac{1}{ \mu^2} , \qquad \dfrac{\alpha^2\rho^2+\nu }{1+\rho^{4}\mu^2} < \dfrac{\alpha^2+\nu }{1+\rho^ {4}\mu^2}< \alpha^2+\nu<\dfrac{1}{\mu^2}(\alpha^2+\nu).
\end{equation*}
\end{proof}

\begin{lmm} 
\label{lemma8}
Let $\bbeta \in \R^n$; $\alpha^2,\nu,t_0 >0$; $0<\mu <1 $ and $R_{\mu}^{i}$ with $i=1,2,3$ given by \eqref{familyR}. So
\begin{equation*}
\left\vert  R_{\mu}^{i}(\bxi) \right\vert < \dfrac{M_i}{\mu^2}, \quad i=1,2,3,
\end{equation*}
where
\begin{equation}
\label{M1}
M_1= \max \left\{2 \,\nu + 2 \,\alpha^2+ \sqrt{n} \left\|\bbeta\right\|_\infty ; \dfrac{2}{ t_0} + \dfrac{\sqrt{n} \left\|\bbeta\right\|_\infty}{ \nu \, t_0}\right\},
\end{equation}
\begin{equation}
\label{M2}
M_2= \max \left\{2 \, \nu + 2 \, \alpha^2+ 2 \,\sqrt{n} \left\|\bbeta\right\|_\infty ; \dfrac{2}{ t_0} + \dfrac{2 \, \sqrt{n} \left\|\bbeta\right\|_\infty}{ \nu \, t_0}\right\},
\end{equation}
\begin{equation}
\label{M3}
M_3= \max \left\{ (\alpha^2+ \nu) \left[8+\frac{4 \, \sqrt{n}\left\|\bbeta\right\|_\infty}{\sqrt{\alpha^2 \, \nu}}\right] ; \dfrac{8}{t_0}+\frac{4 \,\sqrt{n}\left\|\bbeta\right\|_\infty}{ t_0 \sqrt{\alpha^2 \, \nu}} \right\}.
\end{equation}
\end{lmm} 
%
\begin{proof}
Since the three regularization operators given in \eqref{familyR} contain the expression $\Lambda(\bxi)$ defined in \eqref{Lambda}, we begin by bounding this operator. The lemma \ref{lemma1} applied to the absolute value of \eqref{Lambda} is used and we immediately obtain, 
\begin{equation}
\label{D_ineq}
\left\vert \Lambda(\bxi)\right\vert = \left\vert \dfrac{z(\bxi)}{1-e^{-z(\bxi) \, t_0}}\right\vert 
=\left\vert \dfrac{\alpha^2 \left\|\bxi\right\|^2 + i\,\bbeta \cdot \bxi + \nu}{1-e^{-(\alpha^2 \left\|\bxi\right\|^2 + i\,\bbeta \cdot \bxi + \nu) \, t_0}}\right\vert
\le 
\left\vert \dfrac{\alpha^2 \left\|\bxi\right\|^2 + i\,\bbeta \cdot \bxi + \nu}{1-e^{-(\alpha^2 \left\|\bxi\right\|^2 + \nu) \, t_0}}\right\vert.
\end{equation}
\setlength{\leftskip}{0pt}
\setlength{\leftskip}{0pt}

%
\emph{If $(\alpha^2 \left\|\bxi\right\|^2 + \nu) \, t_0 \ge 1$ } : From the inequality \eqref{D_ineq}, using the triangular inequality, the lemma \ref{lemma2} and the lemma \ref{lemma6}, we have,
%

%
\begin{equation} \label{part1_1}
\begin{split}
\phantom{space} \left|R_{\mu}^{1}(\bxi)\right| &= \left|\dfrac{\Lambda(\bxi)}{1+\left\|\bxi\right\|^2\mu^2}\right|\leq\dfrac{\vert \alpha^2 \, \left\|\bxi\right\|^2 + \nu +  \bbeta \cdot \bxi \, i \vert}{(1-e^{-(\alpha^2 \left\|\bxi\right\|^2+\nu ) \, t_0})(1+\left\|\bxi\right\|^2\mu^2)} \\
&< 2 \left( \dfrac{\alpha^2 \left\|\bxi\right\|^2 + \nu }{1 + \left\|\bxi\right\|^2 \mu^2} + \dfrac{\left\|\bbeta\right\| \left\| \xi \right\| }{1 + \left\|\bxi\right\|^2 \mu^2} \right) 
<2 \left( \dfrac{1}{\mu^2} \left(\nu+ \alpha^2\right) +\dfrac{\sqrt{n} \left\|\bbeta\right\|_\infty}{2 \mu^2}\right) \\
& =  \dfrac{1}{\mu^2}\left(2 \nu + 2 \alpha^2+ \sqrt{n} \left\|\bbeta\right\|_\infty \right). 
\end{split}
\end{equation}
Similarly, from the inequality \eqref{D_ineq}, using the triangular inequality, the lemma \ref{lemma2} and the lemma \ref{lemma7}, we obtain,
\begin{equation} \label{part1_2}
\begin{split}
\phantom{space} \left|R_{\mu}^{2}(\bxi)\right| &= \left|\dfrac{\Lambda(\bxi)}{1+\left\|\bxi\right\|^4\mu^2}\right|\leq\dfrac{\vert \alpha^2 \, \left\|\bxi\right\|^2 + \nu +  \bbeta \cdot \bxi \, i \vert}{(1-e^{-(\alpha^2 \left\|\bxi\right\|^2+\nu ) \, t_0})(1+\left\|\bxi\right\|^4\mu^2)} \\
&< 2 \left( \dfrac{\alpha^2 \left\|\bxi\right\|^2 + \nu }{1 + \left\|\bxi\right\|^4 \mu^2} + \dfrac{\left\|\bbeta\right\| \left\| \xi \right\| }{1 + \left\|\bxi\right\|^4 \mu^2} \right) 
<2 \left( \dfrac{1}{\mu^2} \left(\nu+ \alpha^2\right) +\dfrac{\sqrt{n} \left\|\bbeta\right\|_\infty}{\mu^2}\right) \\
& =  \dfrac{2}{\mu^2}\left( \nu +  \alpha^2+  \sqrt{n} \left\|\bbeta\right\|_\infty \right). 
\end{split}
\end{equation}
Finally, from the inequality \eqref{D_ineq}, using the triangular inequality, the lemmas \ref{lemma2}, \ref{lemma3}, \ref{lemma4} and \ref{lemma5} hold that,
\begin{equation} \label{part1_3}
\begin{split}
\phantom{space} \left|R_{\mu}^{3}(\bxi)\right| &= \left|\dfrac{\Lambda(\bxi)}{e^{(\left\|\bxi\right\|^2\mu^2)/4}}\right|\leq\dfrac{\vert \alpha^2 \, \left\|\bxi\right\|^2 + \nu +  \bbeta \cdot \bxi \, i \vert}{(1-e^{-(\alpha^2 \left\|\bxi\right\|^2+\nu ) \, t_0})(e^{(\left\|\bxi\right\|^2\mu^2)/4})} \\
&< 2 \left( \dfrac{\alpha^2 \left\|\bxi\right\|^2 + \nu +\left\|\bbeta\right\| \left\| \xi \right\| }{e^{(\left\|\bxi\right\|^2\mu^2)/4}} \right) \dfrac{\left\|\bxi\right\|^2}{1-e^{-\left\|\bxi\right\|^2}}\dfrac{1-e^{-\left\|\bxi\right\|^2}}{\left\|\bxi\right\|^2}\\
&< \dfrac{8}{\mu^2} (\alpha^2 \left\|\bxi\right\|^2 + \nu) \left( 1 +\frac{\left\|\bbeta\right\| \left\| \xi \right\|}{\alpha^2 \left\|\bxi\right\|^2 + \nu}  \right) \dfrac{1-e^{-\left\|\bxi\right\|^2}}{\left\|\bxi\right\|^2}\\
&= \dfrac{8}{\mu^2} \left(\dfrac{\alpha^2 \left\|\bxi\right\|^2 (1-e^{-\left\|\bxi\right\|^2})}{\left\|\bxi\right\|^2} + \nu \, \dfrac{1-e^{-\left\|\bxi\right\|^2}}{\left\|\bxi\right\|^2}\right) \left( 1 +\frac{\left\|\bbeta\right\| \left\| \xi \right\|}{\alpha^2 \left\|\bxi\right\|^2 + \nu}  \right) \\
& <  \dfrac{8}{\mu^2}(\alpha^2+ \nu) \left[1+\frac{\sqrt{n}\left\|\bbeta\right\|_\infty}{2\sqrt{\alpha^2 \, \nu}}\right]. 
\end{split}
\end{equation}

\setlength{\leftskip}{0pt}
\setlength{\leftskip}{0pt}

%
\emph{If $(\alpha^2 \left\|\bxi\right\|^2 + \nu) \, t_0 \in (0,1)$ } : From the inequality \eqref{D_ineq}, using the triangular inequality, the lemma \ref{lemma2} and the lemma \ref{lemma6}, we have,
%

%
\begin{equation} \label{part2_1}
\begin{split}
\phantom{space} \left|R_{\mu}^{1}(\bxi)\right| &= \left|\dfrac{\Lambda(\bxi)}{1+\left\|\bxi\right\|^2\mu^2}\right| \leq \dfrac{\vert \alpha^2 \, \left\|\bxi\right\|^2 + \nu +  \bbeta \cdot \bxi \, i \vert}{(1-e^{-(\alpha^2 \left\|\bxi\right\|^2+\nu ) \, t_0})(1+\left\|\bxi\right\|^2\mu^2)} \\
&=
\dfrac{(\alpha^2 \, \left\|\bxi\right\|^2 + \nu) \, t_0}{1-e^{-(\alpha^2 \left\|\bxi\right\|^2+\nu) \, t_0}}\left(\dfrac{1}{(1+\left\|\bxi\right\|^2\mu^2) \, t_0} + \dfrac{\left|\bbeta \cdot \bxi \, i \right|}{(1+\left\|\bxi\right\|^2\mu^2)(\alpha^2 \, \left\|\bxi\right\|^2 + \nu) \, t_0} \right) \\
&<
2 \left(\dfrac{1}{ t_0} + \dfrac{\sqrt{n} \left\|\bbeta\right\|_\infty}{2\, \mu^2 \,\nu \, t_0} \right) 
<
\dfrac{2}{\mu^2} \left(\dfrac{1}{ t_0} + \dfrac{\sqrt{n} \left\|\bbeta\right\|_\infty}{2\, \nu \, t_0} \right). 
\end{split}
\end{equation}
Similarly, from the inequality \eqref{D_ineq}, using the triangular inequality, the lemma \ref{lemma2} and the lemma \ref{lemma7}, we obtain,
\begin{equation} \label{part2_2}
\begin{split}
\phantom{space} \left|R_{\mu}^{2}(\bxi)\right| &= \left|\dfrac{\Lambda(\bxi)}{1+\left\|\bxi\right\|^4\mu^2}\right| \leq \dfrac{\vert \alpha^2 \, \left\|\bxi\right\|^2 + \nu +  \bbeta \cdot \bxi \, i \vert}{(1-e^{-(\alpha^2 \left\|\bxi\right\|^2+\nu ) \, t_0})(1+\left\|\bxi\right\|^4\mu^2)} \\
&=
\dfrac{(\alpha^2 \, \left\|\bxi\right\|^2 + \nu) \, t_0}{1-e^{-(\alpha^2 \left\|\bxi\right\|^2+\nu) \, t_0}}\left(\dfrac{1}{(1+\left\|\bxi\right\|^4\mu^2) \, t_0} + \dfrac{\left|\bbeta \cdot \bxi \, i \right|}{(1+\left\|\bxi\right\|^4\mu^2)(\alpha^2 \, \left\|\bxi\right\|^2 + \nu) \, t_0} \right) \\
&<
2 \left(\dfrac{1}{ t_0} + \dfrac{\sqrt{n} \left\|\bbeta\right\|_\infty}{ \mu^2 \,\nu \, t_0} \right) 
<
\dfrac{2}{\mu^2} \left(\dfrac{1}{ t_0} + \dfrac{\sqrt{n} \left\|\bbeta\right\|_\infty}{ \nu \, t_0} \right). 
\end{split}
\end{equation}
Finally, from the inequality \eqref{D_ineq}, using the triangular inequality, the lemmas \ref{lemma2}, \ref{lemma3}, \ref{lemma4} and \ref{lemma5} hold that,
\begin{equation} \label{part3_3}
\begin{split}
\phantom{space} \left|R_{\mu}^{3}(\bxi)\right| &= \left|\dfrac{\Lambda(\bxi)}{e^{(\left\|\bxi\right\|^2\mu^2)/4}}\right|  \leq  \dfrac{\vert \alpha^2 \, \left\|\bxi\right\|^2 + \nu +  \bbeta \cdot \bxi \, i \vert  \, (\alpha^2 \, \left\|\bxi\right\|^2 + \nu) \, t_0}{(1-e^{-(\alpha^2 \left\|\bxi\right\|^2+\nu ) \, t_0}) \, (e^{(\left\|\bxi\right\|^2\mu^2)/4}) \, (\alpha^2 \, \left\|\bxi\right\|^2 + \nu) \, t_0} \\
&< 2 \left(\dfrac{\vert \alpha^2 \, \left\|\bxi\right\|^2 + \nu +  \bbeta \cdot \bxi \, i \vert }{(e^{(\left\|\bxi\right\|^2\mu^2)/4} \, (\alpha^2 \, \left\|\bxi\right\|^2 + \nu) \, t_0}\right) \dfrac{\left\|\bxi\right\|^2}{1-e^{-\left\|\bxi\right\|^2}}\dfrac{1-e^{-\left\|\bxi\right\|^2}}{\left\|\bxi\right\|^2}\\
&< \dfrac{8}{\mu^2}  \left( \dfrac{1}{t_0} +\frac{\left\|\bbeta\right\| \left\| \xi \right\|}{(\alpha^2 \left\|\bxi\right\|^2 + \nu) \, t_0}  \right) \dfrac{1-e^{-\left\|\bxi\right\|^2}}{\left\|\bxi\right\|^2}
< \dfrac{8}{\mu^2}\left[\dfrac{1}{t_0}+\frac{\sqrt{n}\left\|\bbeta\right\|_\infty}{2 \, t_0 \sqrt{\alpha^2 \, \nu}}\right]. 
\end{split}
\end{equation}

To finish the proof of the lemma \ref{lemma8}, all that remains is to combine the expressions \eqref{part1_1}-\eqref{part3_3}.
\end{proof}

\begin{lmm} 
\label{lemma9}
Let $\bxi \in \R^n$ and $0<\mu<1 $ then the following inequality holds
\begin{equation}
\label{cota_gral}
\sup_{\left\|\bxi\right\| \in \R} \left\vert (1+ \left\|\bxi\right\|^2)^{-p/2} \left(1-\dfrac{R_{\mu}^{i}(\bxi)}{\Lambda(\bxi)}\right)\right\vert \le \max \left\{\mu^p,\mu^2\right\}, \quad i=1,2,3,
\end{equation}
where $\Lambda(\bxi)$ and $R_{\mu}^{i}(\bxi)$ are given by \eqref{Lambda} and \eqref{familyR} respectively.
\end{lmm}
\begin{proof}
Let
\begin{equation*}
\Omega_i(\bxi):=(1+ \left\|\bxi\right\|^2)^{-p/2} \left(1-\dfrac{R_{\mu}^{i}(\bxi)}{\Lambda(\bxi)}\right)  , \quad i=1,2,3.
\end{equation*}
Three cases are considered for the proof.

\vspace{0.5cm}

\textbf{Case 1} $ (\left\|\bxi \right\| \geq \left\|\bxi_0\right\|:=\frac{1}{\mu})$. So

		    \begin{equation} 
        \label{caso1}
        \Omega_i(\bxi) \leq (1+ \left\|\bxi\right\|^2)^{-p/2} \leq \left\| \bxi \right\|^{-p} \leq \left\| \bxi_0 \right\|^{-p}= \mu^{p}, \quad i=1,2,3.
        \end{equation}

				
 \textbf{Case 2} $(\left\|\bxi\right\|   < 1)$. So 

		
        \begin{equation} 
        \label{caso2_1}
        \Omega_1(\bxi)= \dfrac{\left\|\bxi\right\|^2 \mu^2}{1+\left\|\bxi\right\|^2 \mu^2} (1+ \left\|\bxi\right\|^2)^{-p/2}\leq \left\|\bxi\right\|^2 \mu^2 (1+ \left\|\bxi\right\|^2)^{-p/2}\leq \mu^2,
        \end{equation}
				\begin{equation} 
        \label{caso2_2}
           \Omega_2(\bxi)= \frac{\left\|\bxi\right\|^4 \mu^2}{1+\left\|\bxi\right\|^4 \mu^2} (1+ \left\|\bxi\right\|^2)^{-p/2}\leq \left\|\bxi\right\|^4 \mu^2 (1+ \left\|\bxi\right\|^2)^{-p/2}\leq \mu^2,
        \end{equation}
				\begin{equation} 
        \label{caso2_3}
              \Omega_3(\bxi)= \left(1-e^{\frac{\left\|\bxi\right\|^2 \mu^2}{4}}\right) (1+ \left\|\bxi\right\|^2)^{-p/2}\leq 1-e^{\frac{\left\|\bxi\right\|^2 \mu^2}{4}}\leq \frac{\left\|\bxi\right\|^2 \mu^2}{4}\leq \mu^2.
        \end{equation}
				
  \textbf{Case 3} $(1 \leq  \left\| \bxi \right\| < \left\|\bxi_0\right\|:=\frac{1}{\mu})$. So 
	
				
         \begin{equation} 
         \label{caso3_1}
        \Omega_1(\bxi)= \dfrac{\left\|\bxi\right\|^2 \mu^2}{1+\left\|\bxi\right\|^2 \mu^2} (1+ \left\|\bxi\right\|^2)^{-p/2}\leq \dfrac{\left\|\bxi\right\|^{2-p} \, \mu^2}{1+\left\|\bxi\right\|^2 \mu^2}\leq \left\|\bxi\right\|^{2-p} \, \mu^2,
         \end{equation}
				 \begin{equation} 
         \label{caso3_2}
        \Omega_2(\bxi)= \dfrac{\left\|\bxi\right\|^4 \mu^2}{1+\left\|\bxi\right\|^4 \mu^2} (1+ \left\|\bxi\right\|^2)^{-p/2}\leq \dfrac{\left\|\bxi\right\|^{2-p} \, \mu^2 \, \left\|\bxi\right\|^2}{1+\left\|\bxi\right\|^4 \mu^2}\leq \dfrac{\left\|\bxi\right\|^{2-p}}{\left\|\bxi\right\|^{4}\,\mu^2} \leq \left\|\bxi\right\|^{2-p} \, \mu^2,
         \end{equation}
				\begin{equation} 
        \label{caso3_3}
              \Omega_3(\bxi)= \left(1-e^{\frac{\left\|\bxi\right\|^2 \mu^2}{4}}\right) (1+ \left\|\bxi\right\|^2)^{-p/2}\leq \left\|\bxi\right\|^{-p} \dfrac{\left\|\bxi\right\|^2 \mu^2}{4} \leq \left\|\bxi\right\|^{2-p} \, \mu^2.
        \end{equation}
				 
If $0<p\leq 2,$ you have 
       \begin{equation} 
         \label{subcaso1_3}
         \Omega_i(\bxi)\leq \left\|\bxi\right\|^{2-p} \, \mu^2  \leq \left\|\bxi_0\right\|^{2-p} \, \mu^2 =\mu^p, \quad i=1,2,3.
         \end{equation}

If $p > 2,$ follow that 
       \begin{equation} 
         \label{subcaso2_3}
        \Omega_i(\bxi)\leq \left\|\bxi\right\|^{2-p} \, \mu^2  \leq \mu^2, \quad i=1,2,3.
         \end{equation}

The expressions given by \eqref{caso1}-\eqref{subcaso2_3} are combined to obtain \eqref{cota_gral}. This concludes the proof.
\end{proof}

\begin{rmrk}
Notice that 
\begin{equation*} 
\lim_{\mu \to 0^+}\dfrac{R_{\mu}^{i}(\bxi)}{\Lambda(\bxi)}=1, \qquad i=1,2,3.
\end{equation*}
\end{rmrk}

\subsubsection{Analytical bound of error}

Now it is possible to obtain a bound for the error estimate.

\begin{defn} 
The norm in Sobolev space $H^p(\R^n)$ for $p>0$ is defined as 
\begin{equation}
\label{Boundf}
\|  f\|_{H^p(\R^n)} := \left(\, \int\limits_{\R^n}{|\widehat{f}|^2 \left(1+\left\|\bxi\right\|^2 \right)^{p}\, d\bxi} \right)^{1/2}. 
\end{equation}
\end{defn}

By means of the definition \eqref{Boundf}, a parameter choice rule is given for the proposed regularizations and an analytical expression is found for the error bound of each estimation.
This can be seen in the following result.

\begin{thrm}[Analytical bound for the estimation error]
\label{boundestimate}
Consider the inverse problem of determining the source $f(\bx)$ in (\ref{transpeqn}). Let $f_{\delta ,\mu}^{i}(\bx)$ with $i=1,2,3$ be the regularization solutions given in \eqref{ffreg2}. It is assumed that there exists $C \in \R^+$ bounding the norm of $f $ in $H^p(\R^n)$ for some $0<p<\infty$ \eqref{Boundf}.
if chosen
\begin{equation}
\label{mureg}
    \mu^2= \left(\dfrac{\delta}{\delta_M}\right)^{2/p+2} < 1,
\end{equation} 

then there exist constants $K_i \, (i=1,2,3)$ independent of $\delta$ such that
\begin{equation*}
\|  f -f_{\delta ,\mu }^{i} \|_{L^2(\R^n)} \le  K_i \, \max \left\{  \left(\dfrac{\delta}{\delta_M}\right)^{2/p+2}, \left(\dfrac{\delta}{\delta_M}\right)^{p/p+2}\right\},  \qquad i=1,2,3.
\end{equation*}
\end{thrm}
\begin{proof}
\begin{equation*}
\left\| \widehat{f}(\bxi )-R_{\mu}^{i} (\bxi) \, \widehat{y}(\bxi ) \right\|_{L^2(\R^n)}
= \left\| \widehat{f}(\bxi) \left(1-\frac{R_{\mu}^{i} (\bxi)}{\Lambda(\bxi)}\right) \frac{(1+ \left\|\bxi\right\|^2)^{p/2}}{(1+ \left\|\bxi\right\|^2)^{p/2}} \right\|_{L^2(\R^n)}, 
\end{equation*}
rearranging
\begin{equation*}
\left\| \widehat{f}(\bxi )-R_{\mu}^{i} (\bxi) \, \widehat{y}(\bxi ) \right\|_{L^2(\R^n)}
\le  \sup_{\left\|\bxi\right\| \in \R} \left| (1+ \left\|\bxi\right\|^2)^{-p/2}\left(1-\frac{R_{\mu}^{i} (\bxi)}{\Lambda(\bxi)}\right)    \right| 		\left\| \widehat{f}(\bxi) (1+ \left\|\bxi\right\|^2)^{p/2} \right\|_{L^2(\R^n)}.
\end{equation*}

We use the definition of the norm in the Sobolev space $(H^p(\R^n))$ given by the expression \eqref{Boundf}, 
\begin{equation}
\label{ineq2teo2}
\left\|\widehat{f}(\bxi )-R_{\mu}^{i} (\bxi) \, \widehat{y}(\bxi ) \right\|_{L^2(\R^n)} 
\le  \sup_{\left\|\bxi\right\| \in \R} \left| (1+ \left\|\bxi\right\|^2)^{-p/2}\left(1-\frac{R_{\mu}^{i} (\bxi)}{\Lambda(\bxi)}\right)    \right|  \left\|  f (\bxi) \right\|_{H^p(\R^n)}.
\end{equation}

By the triangle inequality we have,
\begin{equation}
\label{dnormas}
 \left\| \widehat{f} - \widehat{f_{\delta ,\mu }^{i}} \right\|_{L^2(\R^n)}   \leq \left\|  \widehat{f} - R_{\mu}^{i} (\bxi) \, \widehat{y}(\bxi ) \right\|_{L^2(\R^n)}    +  \left\| R_{\mu}^{i} (\bxi) \, \widehat{y}(\bxi )  - \widehat{f_{\delta ,\mu }^{i}}\right\|_{L^2(\R^n)}.  
\end{equation}

From the inequalities (\ref{ineq2teo2})-(\ref{dnormas}) together with the definition of the regularized source \eqref{ffregtransf} in frequency space, we obtain,
\begin{equation}
\label{eq_aux_dem1}
\left\|  \widehat{f}  -\widehat{f_{\delta ,\mu }^{i}} \right\|_{L^2(\R^n)}
\le  \sup_{\left\|\bxi\right\| \in \R} \left| (1+ \left\|\bxi\right\|^2)^{-p/2}\left(1-\frac{R_{\mu}^{i}(\bxi)}{\Lambda(\bxi)}\right)    \right|  \left\|  f \right\|_{H^p(\R^n)}+\sup_{\left\|\bxi\right\| \in \R} \left|R_{\mu}^{i}(\bxi)\right|\left\|\widehat{y}-\widehat{y}_{\delta}\right\|_{L^2(\R^n)},
\end{equation}
Parseval's identity \cite{Marks09} is used in \eqref{eq_aux_dem1}, the fact that $C$ bounds $f$ in norm $H^p(\R^n)$ and the assumption that the error in the data is bounded $(\left\|y-y_{\delta }\right\|_{L^2(\R^n)}=\left\|\widehat{y}-\widehat{y }_{\delta }\right\|_{L^2(\R^n)} \leq \delta)$, then
\begin{equation}
\label{eq_aux_dem2}
\left\| f -f_{\delta ,\mu }^{i} \right\|_{L^2(\R^n)} =\left\|  \widehat{f} -\widehat{f_{\delta ,\mu }^{i}}  \right\|_{L^2(\R^n)}
\le C \, \sup_{\left\|\bxi\right\| \in \R} \left| (1+ \left\|\bxi\right\|^2)^{-p/2}\left(1-\frac{R_{\mu}^{i}(\bxi)}{\Lambda(\bxi)}\right)    \right|  + \delta \, \sup_{\left\|\bxi\right\| \in \R} \left|R_{\mu}^{i}(\bxi)\right|.
\end{equation}  

By virtue of the lemmas \ref{lemma8}, \ref{lemma9}, the expression \eqref{eq_aux_dem2} can be rewritten as,
\begin{equation}
\label{eq_aux_dem3}
\left\| f -f_{\delta ,\mu }^{i}  \right\|_{L^2(\R^n)} 
\le C \, \max\left\{\mu^2, \mu^p\right\}  + \frac{\delta}{\mu^2} \,M_i,
\end{equation}  
where $M_i$ with $i=1,2,3$ is given by the equations \eqref{M1}, \eqref{M2} and \eqref{M3}. It is used in the equation \eqref{eq_aux_dem3} $\mu^2= \left(\dfrac{\delta}{\delta_M} \right)^{2/p+2}$ \eqref{mureg} and it is obtained,
\begin{equation*}
\left\| f -f_{\delta ,\mu }^{i}  \right\|_{L^2(\R^n)} 
\le C \, \max\left\{\left(\dfrac{\delta}{\delta_M}\right)^{2/p+2}, \left(\dfrac{\delta}{\delta_M}\right)^{p/p+2}\right\}  + \delta_M  \,M_i \, \left(\dfrac{\delta}{\delta_M}\right)^{p/p+2},
\end{equation*} 
equivalently
\begin{equation}
\label{cota_final}
\left\|  f -f_{\delta ,\mu }^{i} \right\|_{L^2(\R^n)} \le  K_i \, \max \left\{  \left(\dfrac{\delta}{\delta_M}\right)^{2/p+2}, \left(\dfrac{\delta}{\delta_M}\right)^{p/p+2}\right\},
\end{equation}
where $K_i=C+\delta_M \,M_i$ with $i=1,2,3$. This concludes the proof.
\end{proof}

\begin{rmrk}
Note that the bound \eqref{cota_final} satisfies
\begin{equation*}
\| f -f_{\delta, \mu}^{i} \|_{L^2(\R^n)} \longrightarrow 0 \,\,\, \text{if} \,\,\, \delta \longrightarrow 0,
  \end{equation*}
which means that the estimate ($f_{\delta, \mu}^{i}$) converges to the source function ($f$) when the noise in the data ($\delta$) tends to $0$.
\end{rmrk}
\begin{rmrk}
Note that the case $p=\infty$ is not included in the hypotheses of the \ref{boundestimate} theorem, the reason is that in this case,
\begin{equation*}
  \| f -f_{\delta, \mu}^{i} \|_{L^2(\R^n)} =\left\| \widehat{f} -\widehat{f_{\delta ,\mu }^{i}} \right\|_{L^2(\R^n)} \le K_i < C \nrightarrow 0.
  \end{equation*}
\end{rmrk}
\begin{rmrk}
Note that the bound obtained for the regularization error \eqref{cota_final} is of H\"older type and only depends on the smoothness of the source and the parameters of the mathematical model.
\end{rmrk}
\begin{nt}
A case of special interest in the bibliography is for $p=2$. In this the expression \eqref{cota_final} is reduced to:
\begin{equation*}
\left\|  f -f_{\delta ,\mu }^{i} \right\|_{L^2(\R^n)} \le  K_i \, \sqrt{\dfrac{\delta}{\delta_M}}.
\end{equation*}
\end{nt}

\subsection{Numerical examples}

We consider concrete examples of the estimation of the source $f$ in $\R^n$ for $n=1,2,3$.
For each of the examples discussed in this section, different values are chosen for the parameters of the source identification problem $( \alpha^2,\bbeta,\nu$, $t_0)$. In addition, to simulate the noise in the data, a set of values of
standard deviation $\epsilon$. The space is discretized into a uniform $n$-dimensional mesh and a data set $\{y_{\delta_1},...,y_{\delta_N}\}$ is obtained from the evaluation of the solution
$u(\bx,t)$ at a fixed time instant given $t_0$ and adding noise, that is,
\begin{equation*}
\label{noisydata}
y_{\delta_i} = y(\bx_i) + \eta_i , \quad i=1,...,N,\,\, \bx_i \in {\cal G},
\end{equation*}
where ${\cal G}$ is a uniform discretization on $\R^n$ chosen and $\eta_i, i=1,...,N$ are realizations of the normally distributed random variable
${\eta}$ with mean 0 and deviation $\epsilon$.
By denoting $y_i=y(\bx_i), i=1,.., N$ and taking into account the noise level $\delta$ satisfying \eqref{noiselevel}, the error
\begin{equation*}
\label{discrerror}
  y -y_{\delta} =(y_1-y_{\delta_1},...,y_N-y_{\delta_N})=(\eta_1,...,\eta_N)
  \end{equation*}
 It is numerically calculated using the Simpson integration method. This calculation allows obtaining an approximate value of the error $y -y_{\delta}$ that directly depends
of the noise. It can be seen that the noise level $\delta$ is a function of the standard deviation $\epsilon$, that is, $\delta=\delta(\epsilon)$.

In practice, the maximum tolerance value for the error in the data $\delta_M$ given in \eqref{noiselevel} is obtained from the calibration, instrumentation and measurement errors of the instruments used to perform the measurements. Here, to make sure that the regularization parameter is less than 1, it is chosen as the maximum value of $\delta$ plus one unit. That is to say,
\begin{equation*}
  \delta_M= 1+ \max \{\delta_1,...,\delta_N\}.
\end{equation*}

Then $\{\widehat{y}_{\delta_1},...,\widehat{y}_{\delta_N}\}$ is calculated using the FFT transform (Fast Fourier Transform) \cite{Van92} and get the regularization solutions
$f^{i}_{\delta ,\mu}$ with $i=1,2,3$ given in \eqref{ffreg2} using the anti-fast Fourier transform \cite{Elden00,Van92}, where the regularization parameter $\mu$ is chosen according to \eqref{mureg},
that is to say,
\begin{equation*}
  \mu^2= \ds \left(\frac{\delta}{\delta_M}\right)^{2/p+2}.
  \end{equation*}
For each of the examples considered, the results of the estimated sources without regularization and those obtained after using the regularization methods presented in this article are graphed. An error table is also included for each example. These tables consider the relative errors made, in the estimation, by each one of the regularization operators used to address the ill-posed problem.
The disturbances used for the construction of the tables are  $\{\epsilon_1;...;\epsilon_5\}= \{10^{-1}, 10^{-2}, 10^{-3}; 10^{-4}; 10^{-5}\}$, these were chosen based on the values of the solution.
	
\begin{nt} 
The reader may notice that in some examples the $p$ considered for the recovery does not correspond to the space $H^p(\R^n)$ to which the source belongs, since $f$ is not in that Hilbert space. These examples were included to show that recovery is reasonable even in those cases.  
\end{nt}

\subsubsection{Examples 1D} \label{subsec:Ej_1D}

In this subsection, three examples of estimation of one-dimensional sources with different characteristics are discussed.
\begin{xmpl}
\label{example1}

For this example, the following parameters are considered $\alpha^2=2 \times 10^{-5}$; $\bbeta= 1 \times 10^{-5}$; $\nu=$1; $N=1001$ and $t_0=5$.
Also, the perturbation values used are $\epsilon \in \{0.2 ; 0.15; 0.1; $0.05\}.
Finally, the source to estimate in this case is:

\begin{equation}
\label{fuente_ex1}
f(x)=\begin{cases}
-1, \qquad & -10 \leq x<-5, \\
  1, \qquad & -5 \leq x <0, \\
  -1, \qquad & 0 \leq x <5, \\
1, \qquad & 5 \leq x \leq 10, \\
  0, \qquad & \text{in another case}.
\end{cases}
\end{equation}
\end{xmpl}

The source \eqref{fuente_ex1} retrieved in example \ref{example1} is of interest in signal theory.
This type of function turns out to be a common example in \cite{Yang11,Yang14} source retrieval problems. This is mainly due to the fact that as it is a discontinuous function, the Gibbs \cite{Elden00,Van92} phenomenon can appear. This phenomenon indicates that when a Fourier series function is developed and it is not continuous in the considered region, it is possible that there is not a good precision in the neighborhoods of the discontinuities.

As can be seen in the graphs of figure
\ref{CurvaCuadrada}, the regularization operators used smooth the inverse operator making the recovery
stable. As expected, the estimation improves for smaller values of perturbation $\epsilon$. In addition, no significant differences between the different operators are noticeable at first glance, although the operator $R^3_{\mu}$ seems to have a better performance in the neighborhoods of discontinuity points.

\vspace{-0.3cm}
\begin{figure}[H]
\begin{center}
\includegraphics[width=0.46\textwidth]{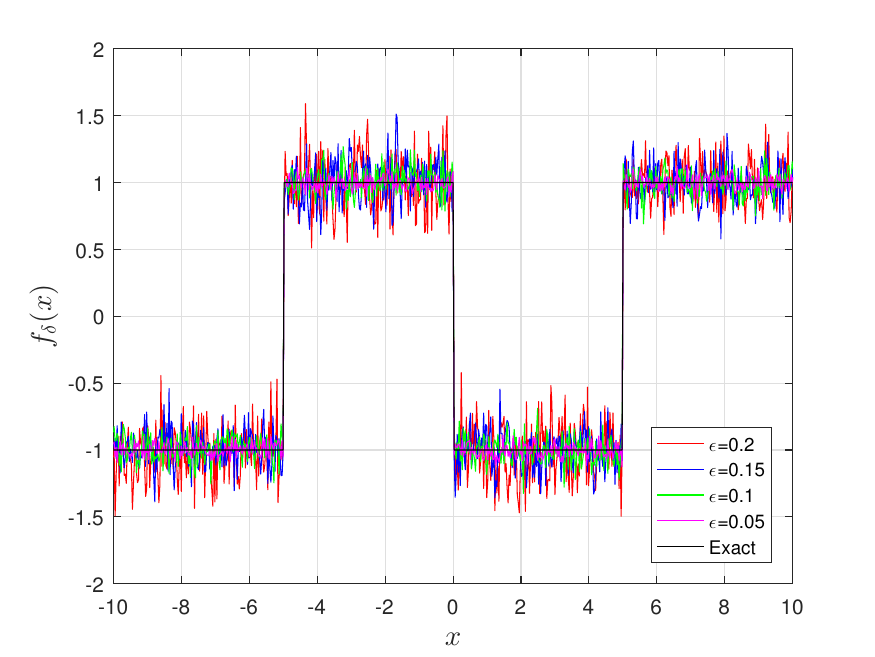}
\includegraphics[width=0.46\textwidth]{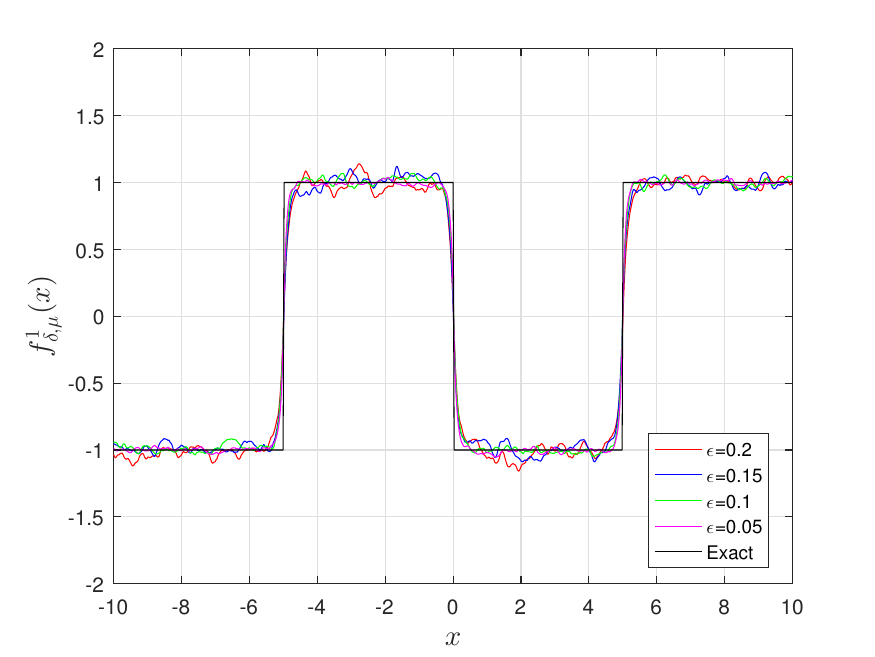}
\includegraphics[width=0.46\textwidth]{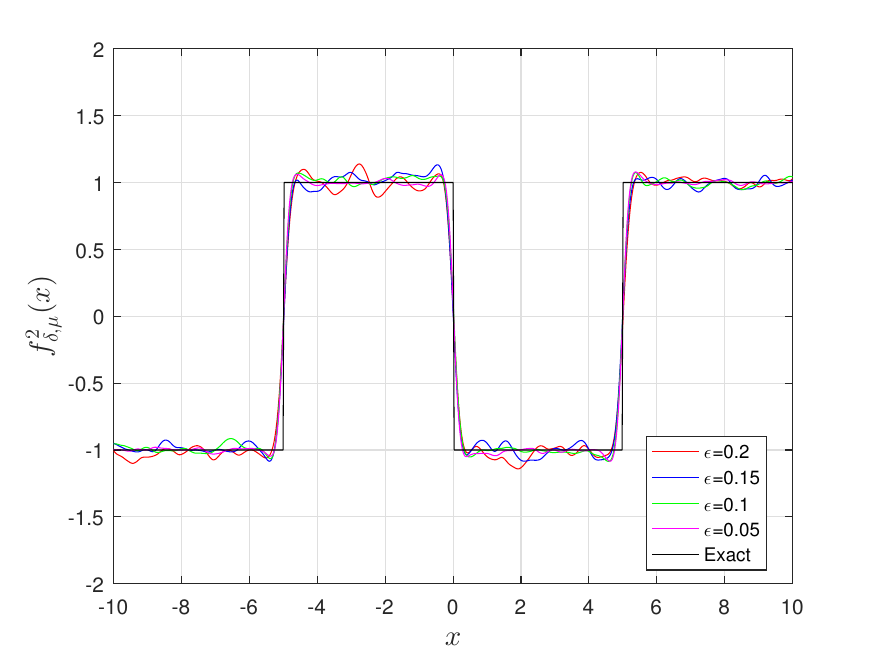}
\includegraphics[width=0.46\textwidth]{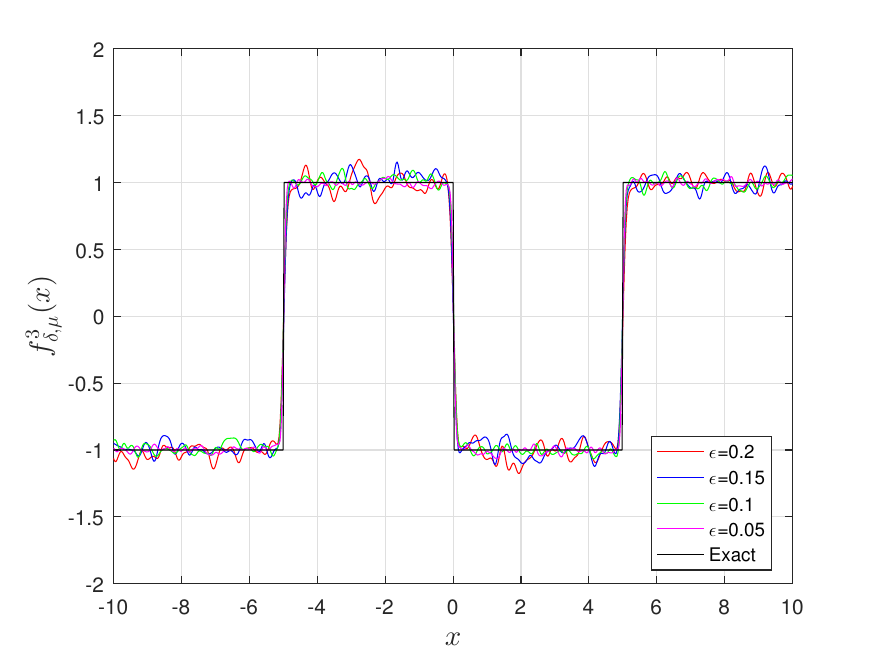}
\vspace{-0.2cm}
\caption{Example \ref{example1}: Non-regularized font with $t_0=5$ (top-left); Regularized sources with $t_0=5$ and $p=1$, using $R^{1}_\mu$ (top-right), $R^{2}_\mu$ (bottom-left), $R ^{3}_\mu$ (bottom-right) for different noise levels.}
\vspace{-0.3cm}
\label{CurvaCuadrada}
\end{center}
\end{figure}

The table \ref{tableej1} includes the relative estimation errors that also show the goodness of the methods used. As noted, for this example, the operator $R^{3}_{\mu}$ performs better.
For example, for the case $\epsilon=10^{-3}$, the relative error of the estimate of the regularized source with $R^{1}_{\mu}$ is $2.74 \%$, with $R^{2}_{\mu}$ is $8.08 \%$ and with $R^{3}_{\mu}$ it is $1.15 \%$.
\begin{table}[H]
\begin{center}
{\begin{tabular}{c}\toprule
Relative errors\\
{\begin{tabular}{lccc} \toprule
\,\, $\epsilon$ &  $ \| f-f^{1}_{\delta,\mu}\|/ \left\|f\right\|$   & $ \| f-f^{2}_{\delta,\mu}\|/ \left\|f\right\|$  & $ \| f-f^{3}_{\delta,\mu}\|/ \left\|f\right\|$   \\ \midrule
$10^{-1}$ & 0.1310  & 0.1472  & 0.1137 \\
$10^{-2}$ & 0.0888  & 0.1217   & 0.0706 \\
$10^{-3}$  & 0.0549 &	0.0996  & 0.0365\\
$10^{-4}$  & 0.0274  & 0.0808  & 0.0115   \\
$10^{-5}$  & 0.0090  & 0.0635 & 0.0026 \\\bottomrule
\end{tabular}}
\end{tabular}}
\end{center}
\vspace{-0.30cm}
\caption{Example \ref{example1}: Relative estimation errors for $ t_0 = 1 $ and $ p = 1.$}
\vspace{-0.30cm}
\label{tableej1}
\end{table}


\begin{xmpl}
\label{example2}

For this example, the following parameters are considered $\alpha^2=1$; $\bbeta= 0 $; $\nu= $0; $N=1001$ and $t_0=0.1$.
Furthermore, the perturbation values used are $\epsilon \in \{0.01 ; 0.007; 0.003; $0.001\}.
Finally, the source to estimate in this case is:

\begin{equation}
\label{fuente_ex2}
f(x)=\begin{cases} 
\left(- \frac{x^3}{4} + \frac{3x}{2}\right) e^{\frac{-x^2}{4}}, \qquad & -10 \leq x \leq 10, \\
 0, \qquad & \text{in another case}. 
\end{cases}
\end{equation}
\end{xmpl}

The source \eqref{fuente_ex2} recovered in example \ref{example2} is of interest in heat transfer problems. See for example, \cite{Dou09b}. This type of function turns out to be a common example in source recovery problems in heat transfer processes \cite{Yang11,Yang14}. The example \ref{example2} is interesting, since this type of sources have already been recovered by means of different approaches in simpler problems where it is used, instead of the full parabolic equation, the heat equation where only diffusion is taken into account \cite{Yan08,Yan10,Yan09}.

Table \ref{tableej2} includes the relative estimation errors that also show the goodness of the methods used. As noted, for this example, the operator $R^{2}_{\mu}$ performs better.
For example, for the case $\epsilon=10^{-3}$, the relative error of the estimate of the regularized source with $R^{1}_{\mu}$ is $5.89 \%$, with $R^{2}_{\mu}$ is $4.82 \%$ and with $R^{3}_{\mu}$ it is $6.48 \%$.

As it can be seen in the graphs of figure \ref{CurvaRara}, the regularization operators used, again, smooth the inverse operator making the recovery stable. In this case, the stabilization effect is much more noticeable to the naked eye (because smaller values of disturbances were considered); due to this and for comparative purposes, a smaller scale was taken for the sources retrieved with the regularization operators.
Once again it can be seen that the estimation improves for smaller values of disturbance $\epsilon$. Also, no significant differences between the different operators are noticeable to the naked eye, although the operator $R^2_{\mu}$ seems to have, in this example, a better overall performance.

\begin{table}[H]
\begin{center}
{\begin{tabular}{c}\toprule
Relative errors\\
{\begin{tabular}{lccc} \toprule
\,\, $\epsilon$ &  $ \| f-f^{1}_{\delta,\mu}\|/ \left\|f\right\|$   & $ \| f-f^{2}_{\delta,\mu}\|/ \left\|f\right\|$  & $ \| f-f^{3}_{\delta,\mu}\|/ \left\|f\right\|$   \\ \midrule
$10^{-1}$  & 0.8387  & 0.6170  & 0.9686 \\
$10^{-2}$   & 0.1746  & 0.1074   & 0.1822 \\
$10^{-3}$   & 0.0589 &	0.0482  & 0.0648\\
$10^{-4}$   & 0.0474  & 0.0471  & 0.0489   \\
$10^{-5}$   & 0.0365  & 0.0321 & 0.0389 \\\bottomrule
\end{tabular}}
\end{tabular}}
\end{center}
\vspace{-0.30cm}
\caption{Example \ref{example2}: Relative estimation errors for $ t_0 = 1 $ and $ p = 1.$}
\vspace{-0.30cm}
\label{tableej2}
\end{table}

\vspace{-0.5cm}
\begin{figure}[H]
\begin{center}
\includegraphics[width=0.46\textwidth]{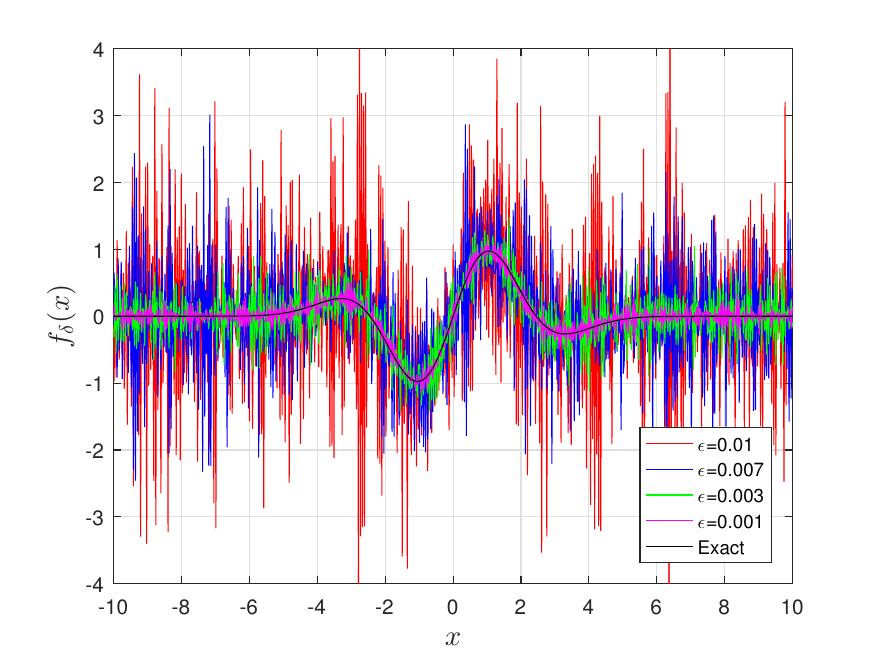}
\includegraphics[width=0.46\textwidth]{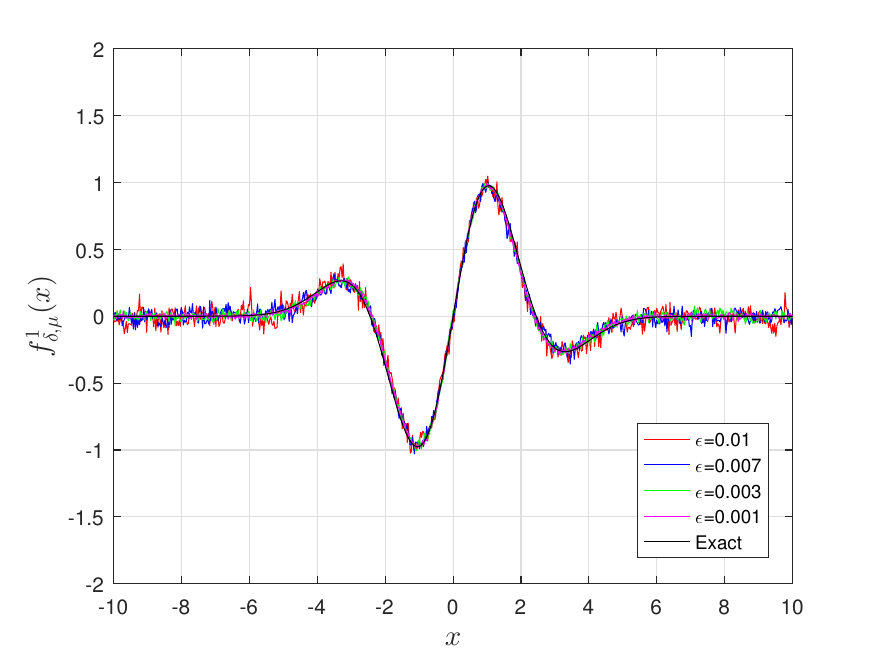}
\includegraphics[width=0.46\textwidth]{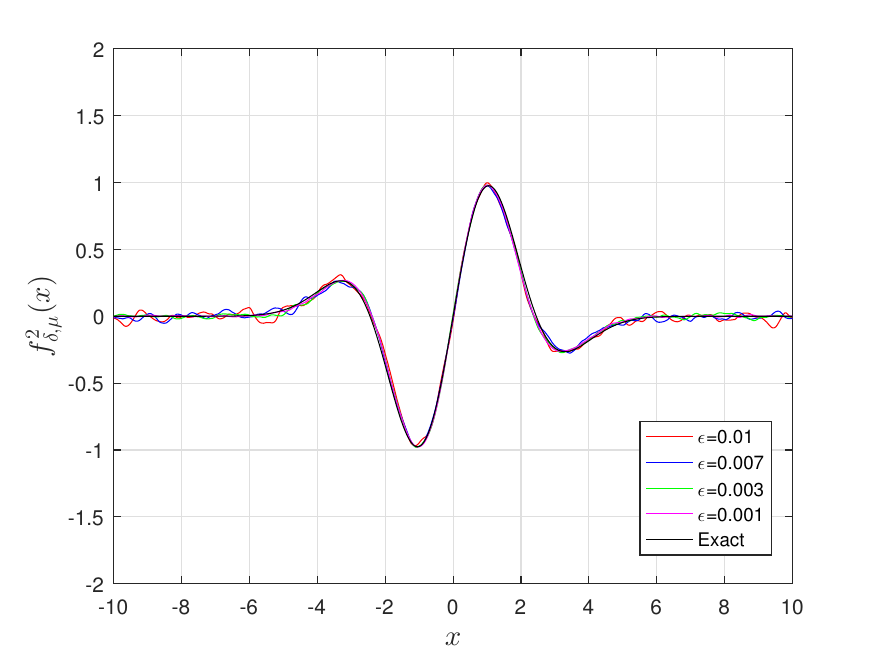}
\includegraphics[width=0.46\textwidth]{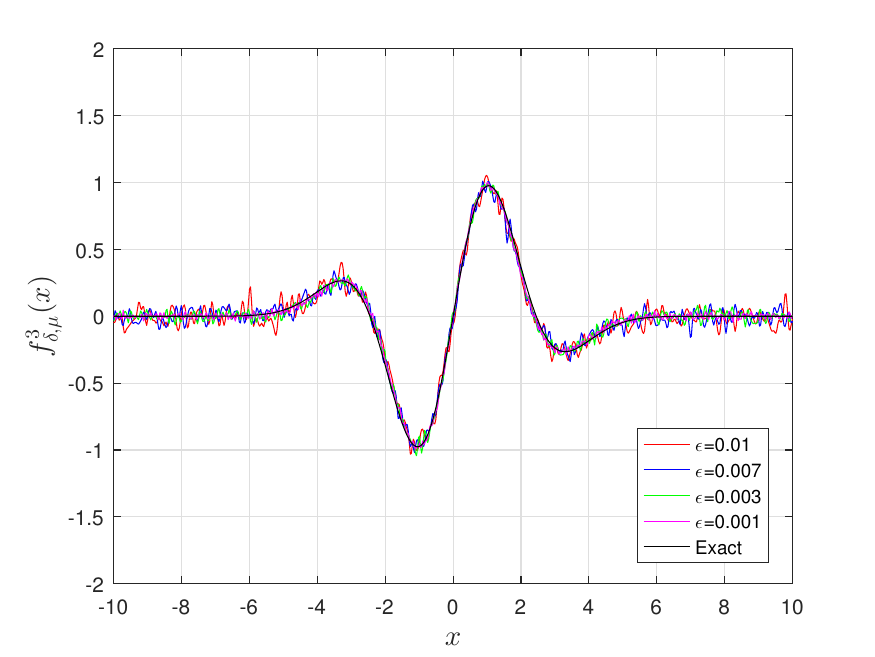}
\vspace{-0.2cm} 
\caption{Ejemplo \ref{example2}: Non-regularized font with $t_0=0.1$ (top-left); Regularized sources with $t_0=0.1$ and $p=2$, using $R^{1}_\mu$ (top-right), $R^{2}_\mu$ (bottom-left), $R^{3}_\mu$ (bottom-right) for different noise levels.}
\vspace{-0.3cm}
\label{CurvaRara}
\end{center}
\end{figure}


\begin{xmpl}
\label{example3}

For this example, the following parameters are considered $\alpha^2=2$; $\bbeta= 0 $; $\nu= 1$; $N=1001$ and $t_0=0.2$.
Furthermore, the perturbation values used are $\epsilon \in \{0.004 ; 0.003; 0.002; $0.0001\}.
Finally, the source to estimate in this case is:

\begin{equation}
\label{fuente_ex3}
f(x)=\begin{cases} 
x+1, \qquad & -1 \leq x <0, \\
-x+1, \qquad & 0 \leq x \leq 1, \\
 0, \qquad & \text{in another case}. 
\end{cases}
\end{equation}
\end{xmpl}

The source \eqref{fuente_ex3} that is recovered in the example \ref{example3} is of interest in various problems with different characteristics due to the particularities
that this function presents. Although it is a continuous function, it has a point where it is not differentiable. Recovery is often difficult at these points. Because of this,
this type of function also turns out to be a common example in source recovery problems, both in signal theory and in heat transfer problems \cite{Yang11,Yang14}.

Table \ref{tableej3} includes the relative estimation errors that again show the benefits of using the methods presented in this article. As noted, for this example, the $R^{2}_{\mu}$ operator performs better.
For example, for the case $\epsilon=10^{-3}$, the relative error of the estimate of the regularized source with $R^{1}_{\mu}$ is $5.35 \%$, with $R^{2}_{\mu}$ is $2.28 \%$ and with $R^{3}_{\mu}$ it is $3.73 \%$.

\begin{table}[H]
\begin{center}
{\begin{tabular}{c}\toprule
Relative errors\\
{\begin{tabular}{lccc} \toprule
\,\, $\epsilon$ &  $ \| f-f^{1}_{\delta,\mu}\|/ \left\|f\right\|$   & $ \| f-f^{2}_{\delta,\mu}\|/ \left\|f\right\|$  & $ \| f-f^{3}_{\delta,\mu}\|/ \left\|f\right\|$   \\ \midrule
$10^{-1}$  & 0.9157  & 0.3027  & 0.6716 \\
$10^{-2}$   & 0.1607  & 0.0353   & 0.1772 \\
$10^{-3}$   & 0.0535 &	0.0228  & 0.0373\\ 
$10^{-4}$   & 0.0478  & 0.0221  &  0.0309  \\
$10^{-5}$   & 0.0239  & 0.0187 & 0.0277 \\\bottomrule
\end{tabular}}
\end{tabular}}
\end{center}
\vspace{-0.30cm}
\caption{Example \ref{example3}: Relative estimation errors for $ t_0 = 1 $ and $ p = 1.$}
\vspace{-0.30cm}
\label{tableej3}
\end{table}

As can be seen in the graphs of figure\ref{CurvaTria}, the regularization operators used, again, smooth the inverse operator making the recovery
stable. In this case the stabilization effect, too, is more noticeable to the naked eye than in the example \ref{example1}. This is because small values of perturbations were considered in order to be able to compare the regularization solutions with the non-regularized ones (which, again, presented many fluctuations). Due to this reason and for comparative purposes, a smaller scale was taken for the sources retrieved with the regularization operators.
Once again it can be seen that the estimation improves for smaller values of disturbance $\epsilon$. Furthermore, no significant differences between the different operators are noticeable to the naked eye, although the operator $R^2_{\mu}$ clearly provides a better estimate.

\vspace{-0.3cm}
\begin{figure}[H]
\begin{center}
\includegraphics[width=0.46\textwidth]{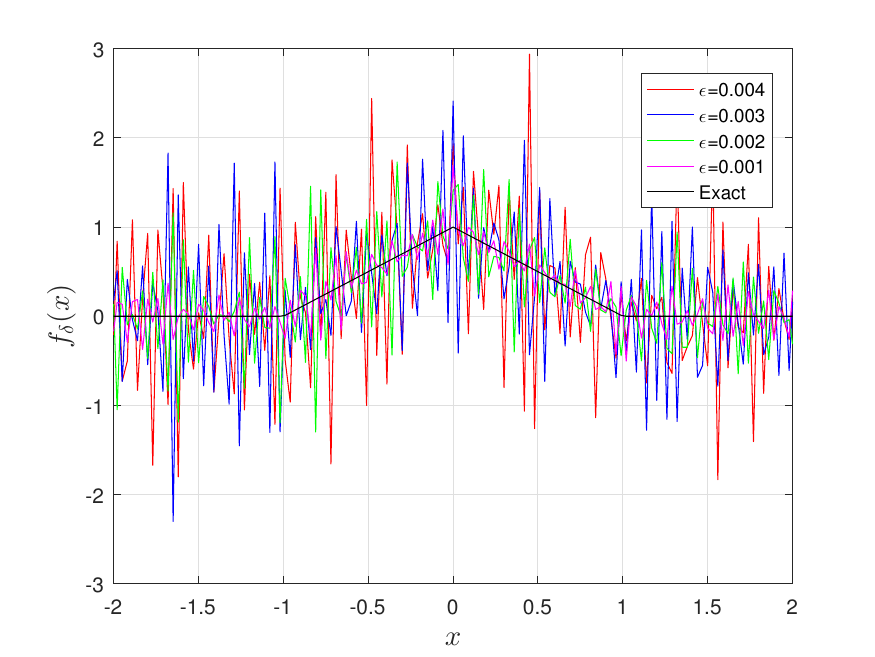}
\includegraphics[width=0.46\textwidth]{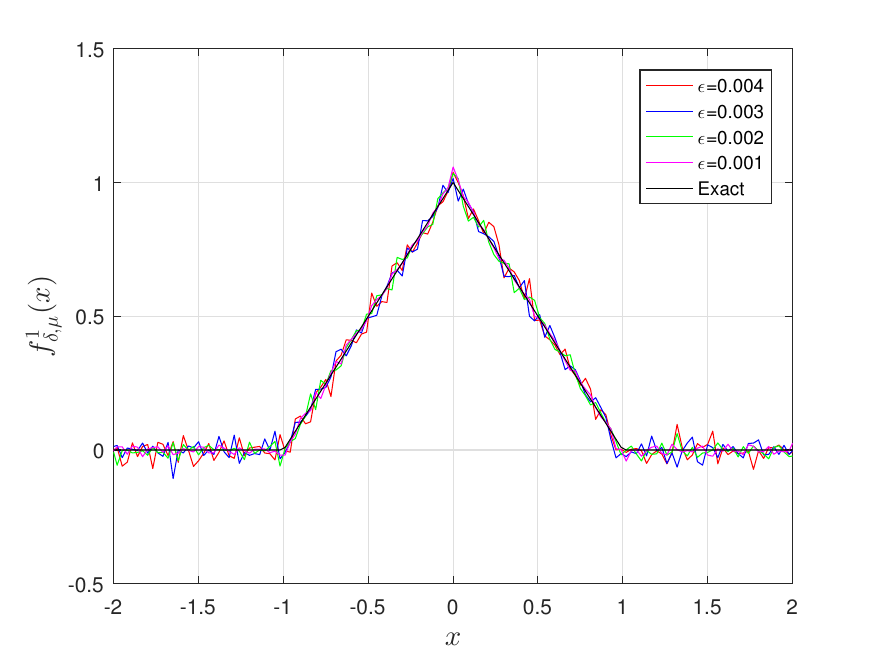}
\includegraphics[width=0.46\textwidth]{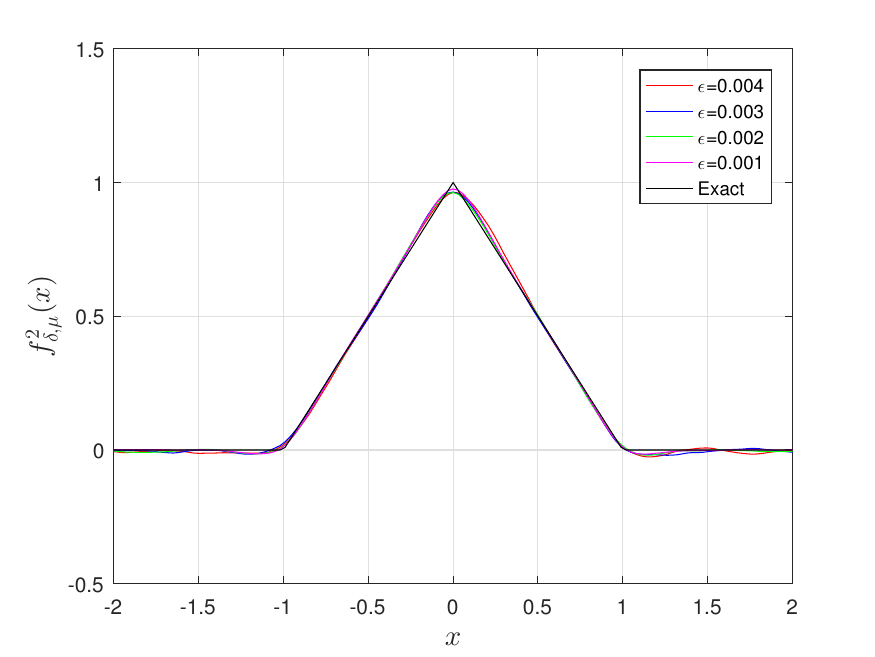}
\includegraphics[width=0.46\textwidth]{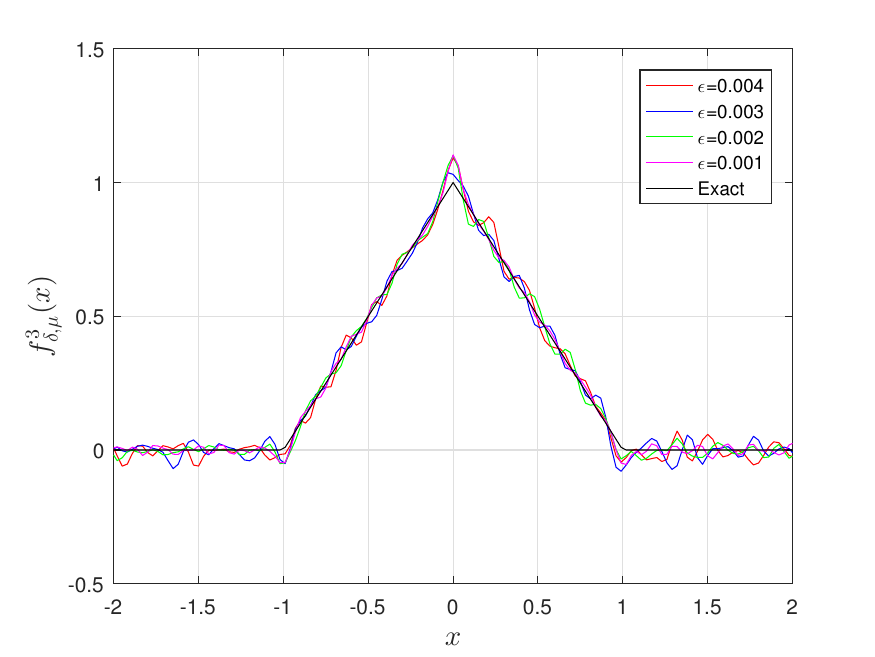}
\vspace{-0.5cm} 
\caption{Ejemplo \ref{example3}: Non-regularized font with $t_0=0.2$ (top-left); Regularized sources with $t_0=0.2$ and $p=4$, using $R^{1}_\mu$ (top-right), $R^{2}_\mu$ (bottom-left), $R^{3}_\mu$ (bottom-right) for different noise levels.}
\vspace{-0.3cm}
\label{CurvaTria}
\end{center}
\end{figure}

\subsubsection{Examples 2D}

Here are two examples of two-dimensional estimation with different characteristics. In order to visualize more clearly the improvements obtained with the regularization operators, not only the recovered sources will be graphed, but also contour lines will be made. These schemes will allow direct comparisons to be made on the sources obtained with the different regularization strategies.

\begin{xmpl}
\label{example4}

For this example, the following parameters are considered $\alpha^2=0.2$; $\bbeta= (0,0) $; $\nu= 0.99$; $N=1001 \times 1001$; $t_0=1$ and $\epsilon=0.025$. 
Finally, the source to estimate in this case is:

\begin{equation}
\label{fuente_ex4}
f(x,y)=\begin{cases} 
\cos\left(\dfrac{x}{20}\right) \, \cos\left(\dfrac{y}{20}\right), \qquad & -40 \leq x,y \leq 40, \\
 0, \qquad & \text{in another case}. 
\end{cases}
\end{equation}
\end{xmpl}

The source \eqref{fuente_ex4} that is recovered in the example \ref{example4} is of interest in various problems with different characteristics due to the particularities that this function presents. Its graph is a surface, smooth, continuous and differentiable.

As can be seen in the graphs of figure \ref{SupSuave}, the regularization operators used, again, smooth the inverse operator making the recovery stable. No significant differences between the different operators are noticeable with the naked eye, although if the level curves are visualized and compared, the operator $R^2_{\mu}$ provides a better estimate.

Table \ref{tableej4} includes the relative estimation errors that once again show the good performance obtained with the methods introduced in this article. As noted, for this example, the operator $R^{2}_{\mu}$ performs better overall.
For example, for the case $\epsilon=10^{-2}$, the relative error of the estimate of the source regularized with $R^{1}_{\mu}$ is $1.18 \%$, with $R^{2}_{\mu}$ is $0.67 \%$ and with $R^{3}_{\mu}$ it is $1.41 \%$.

\begin{table}[H]
\begin{center}
{\begin{tabular}{c}\toprule
Relative errors\\
{\begin{tabular}{lccc} \toprule
\,\, $\epsilon$ &  $ \| f-f^{1}_{\delta,\mu}\|/ \left\|f\right\|$   & $ \| f-f^{2}_{\delta,\mu}\|/ \left\|f\right\|$  & $ \| f-f^{3}_{\delta,\mu}\|/ \left\|f\right\|$   \\ \midrule
$10^{-1}$  & 0.0846  & 0.0583  & 0.0989 \\
$10^{-2}$   & 0.0118  & 0.0067   & 0.0141 \\
$10^{-3}$   & 0.0029 &	0.0011  & 0.0045\\ 
$10^{-4}$   & 0.0008  & 0.0004  &  0.0013  \\
$10^{-5}$   & 0.0004  & 0.0003 & 0.0004 \\\bottomrule
\end{tabular}}
\end{tabular}}
\end{center}
\vspace{-0.30cm}
\caption{Example \ref{example4}: Relative estimation errors for $ t_0 = 1 $ and $ p = 1.$}
\vspace{-0.30cm}
\label{tableej4}
\end{table}


\begin{xmpl}
\label{example5}

For this example, the following parameters are considered $\alpha^2=1$; $\bbeta= (0,0) $; $\nu= 1$; $N=1001 \times 1001$; $t_0=0.4$ and $\epsilon=0.05$. 
Finally, the source to estimate in this case is:

\begin{equation}
\label{fuente_ex5}
f(x,y)=\begin{cases} 
10+x-y, \qquad & -10 \leq x \leq 0, \qquad  0 \leq y \leq 10+x, \\
10+x+y, \qquad & -10 \leq x \leq 0, \qquad  -10-x \leq y \leq 0, \\
10-x-y, \qquad & 0 \leq x \leq 10, \qquad  0 \leq y \leq 10-x, \\
10-x+y, \qquad & 0 \leq x \leq 10, \qquad  -10+x \leq y \leq 0, \\
 0, \qquad & \text{in another case}. 
\end{cases}
\end{equation}
\end{xmpl}

\begin{figure}[H]
\begin{center}
\includegraphics[width=0.37\textwidth]{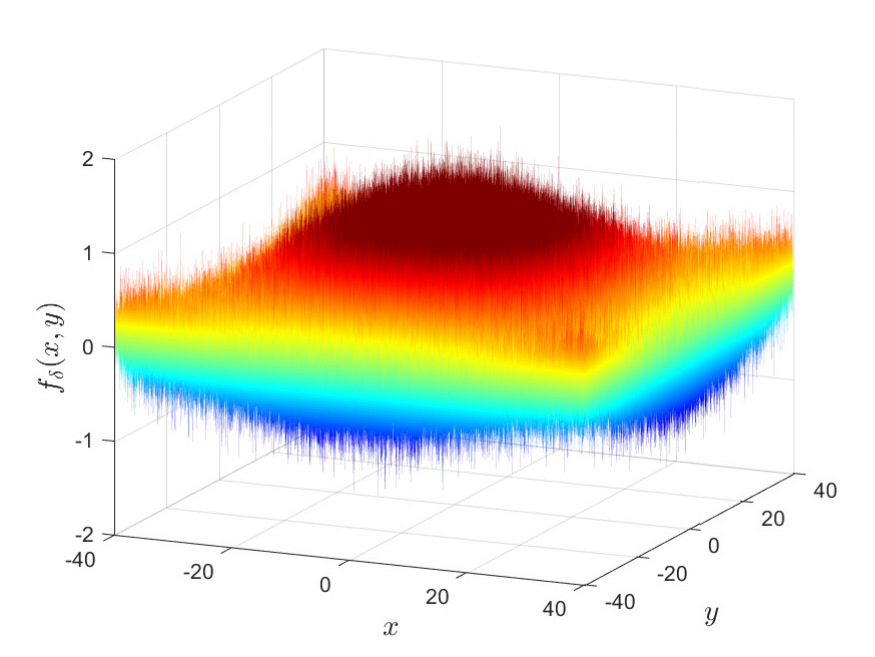}
\includegraphics[width=0.37\textwidth]{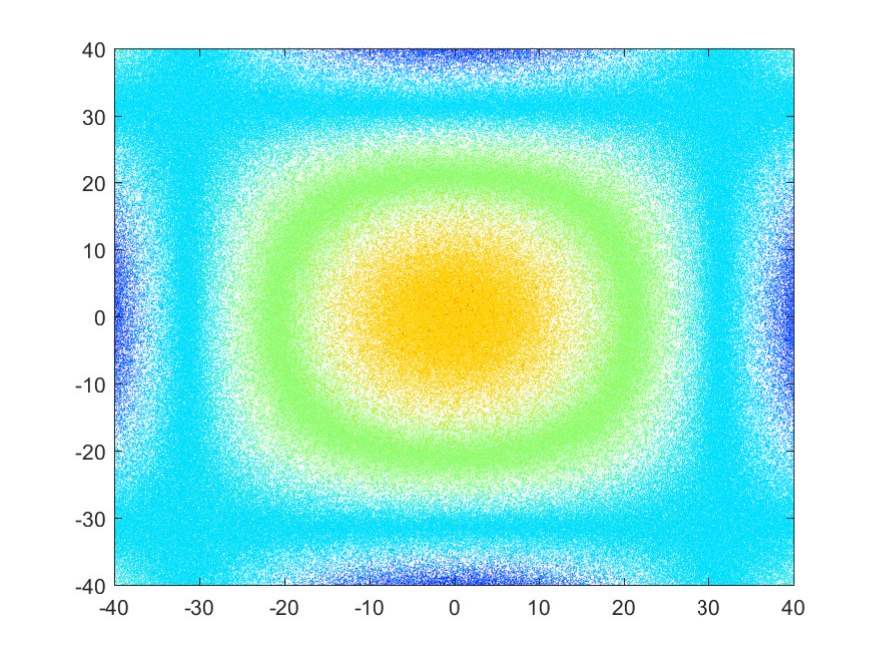}
\includegraphics[width=0.37\textwidth]{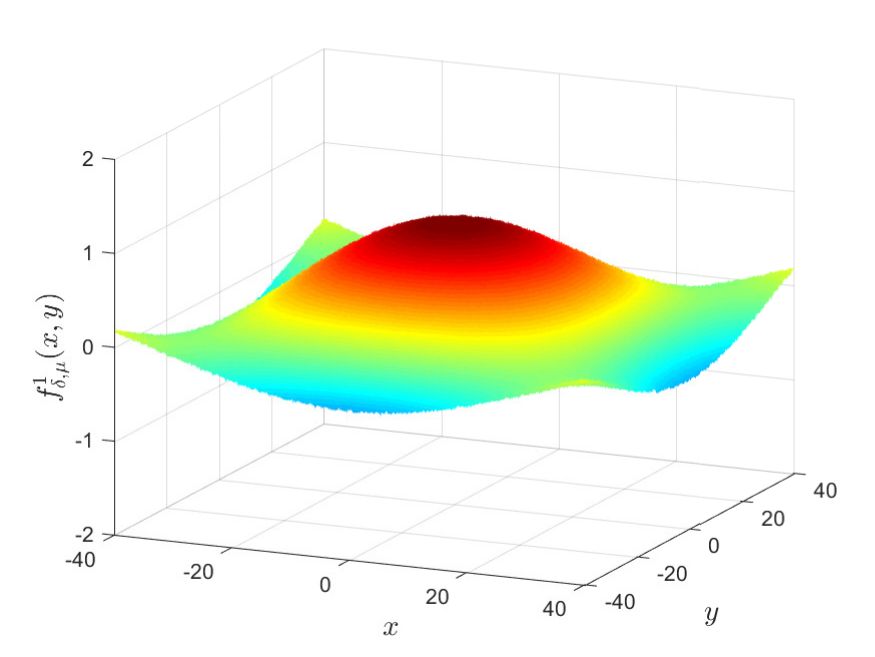}
\includegraphics[width=0.37\textwidth]{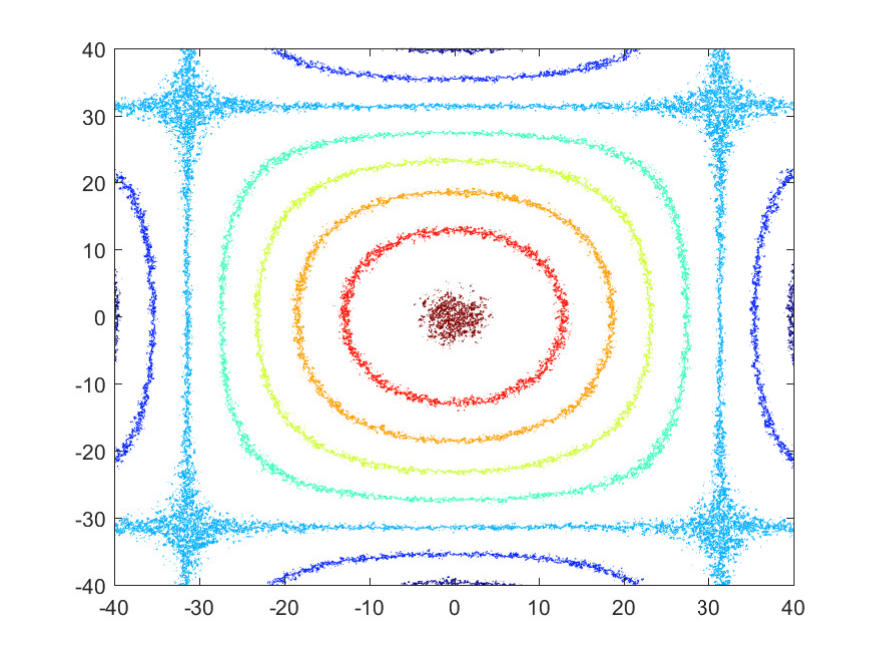}
\includegraphics[width=0.37\textwidth]{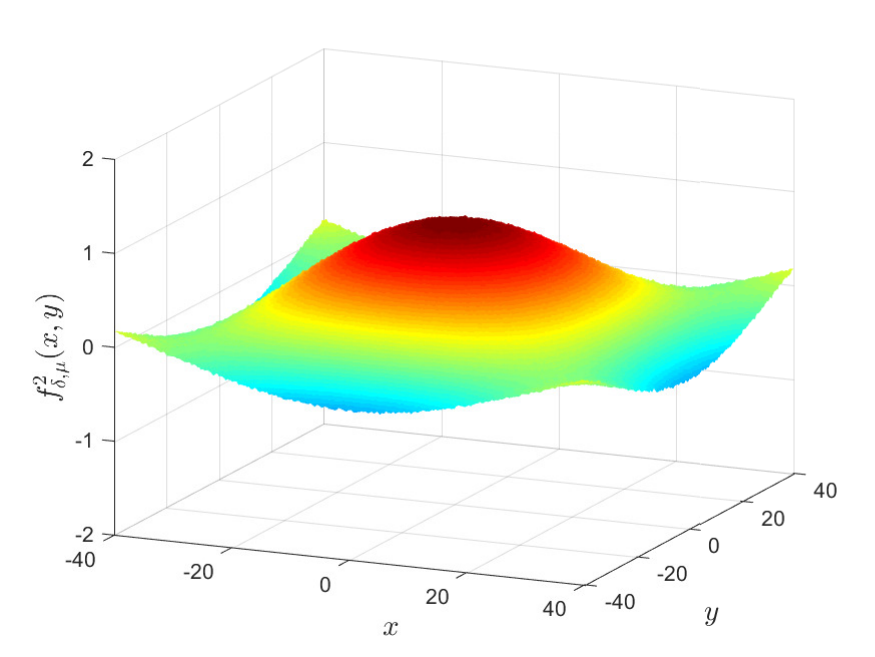}
\includegraphics[width=0.37\textwidth]{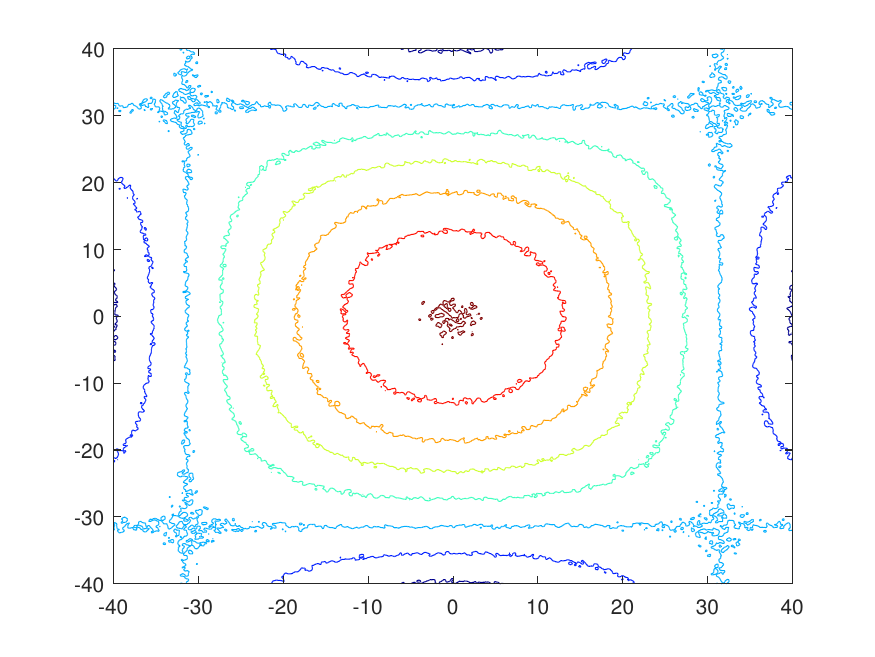}
\includegraphics[width=0.37\textwidth]{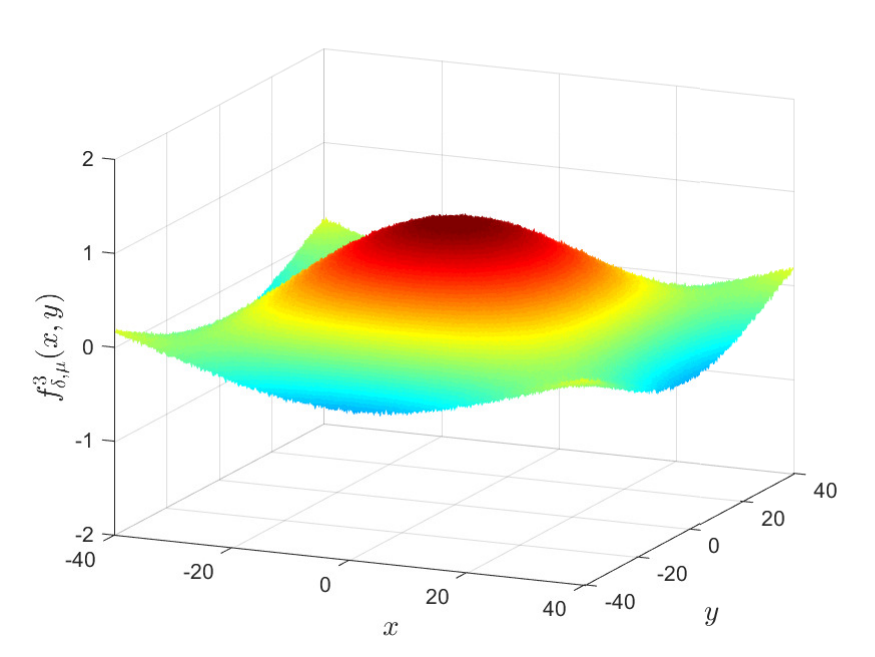}
\includegraphics[width=0.37\textwidth]{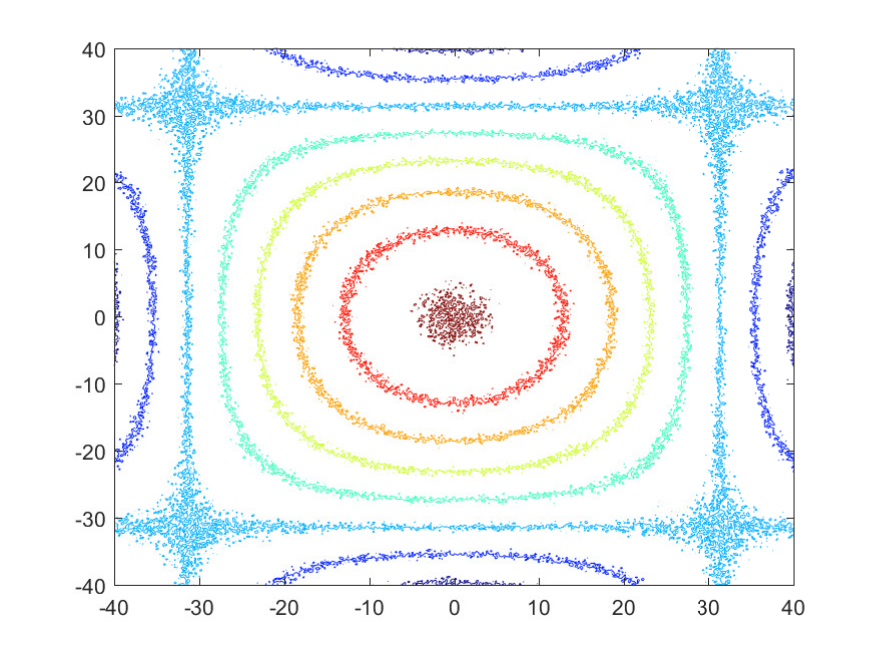}
\vspace{-0.8cm} 
\caption{Example \ref{example4}: Fonts and contour lines. Not regularized with $t_0=1$ (First row); regularized with $t_0=1$ and $p=1$, using $R^{1}_\mu$ (second row), $R^{2}_\mu$ (third row), $R^{3 }_\mu$ (fourth row) for $\epsilon=0.025$.}
\vspace{-0.8cm}
\label{SupSuave}
\end{center}
\end{figure}

\vspace{-0.5cm}
\begin{figure}[H]
\begin{center}
\includegraphics[width=0.37\textwidth]{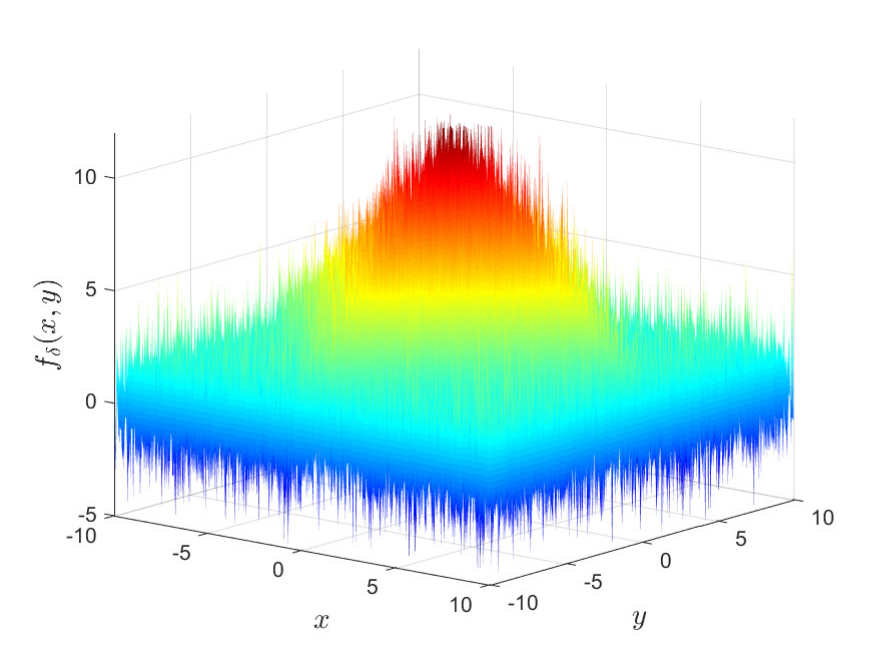}
\includegraphics[width=0.37\textwidth]{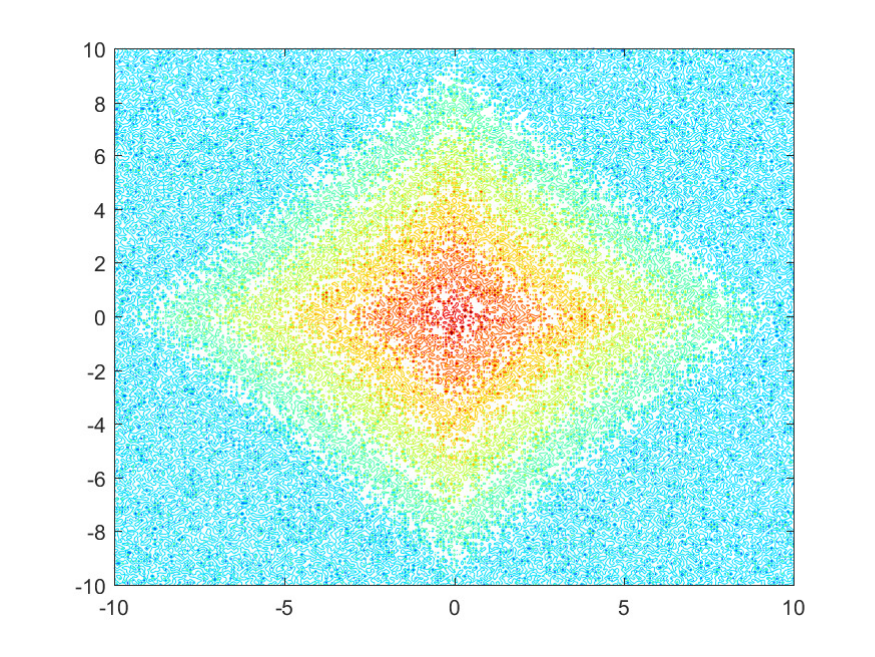}
\includegraphics[width=0.37\textwidth]{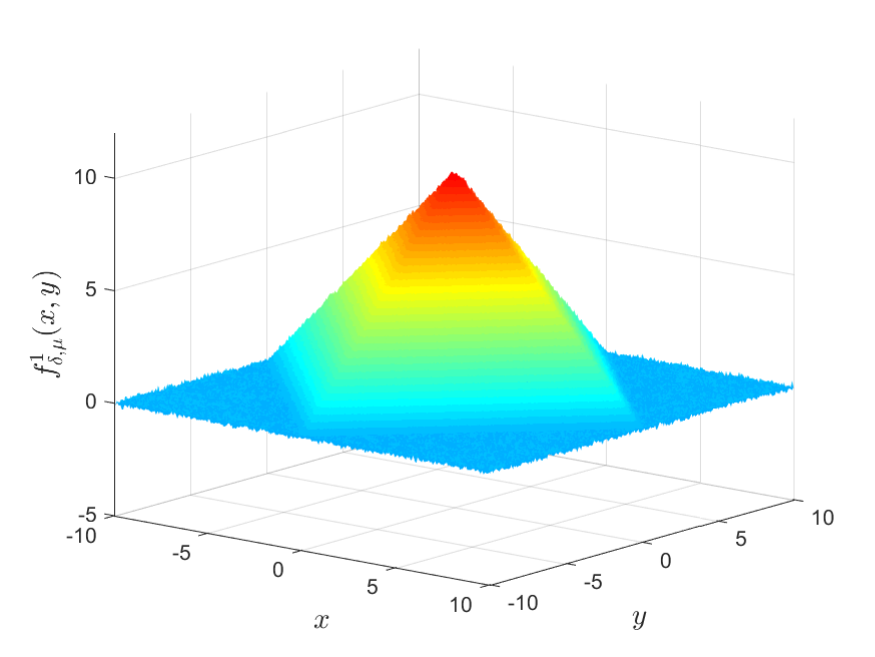}
\includegraphics[width=0.37\textwidth]{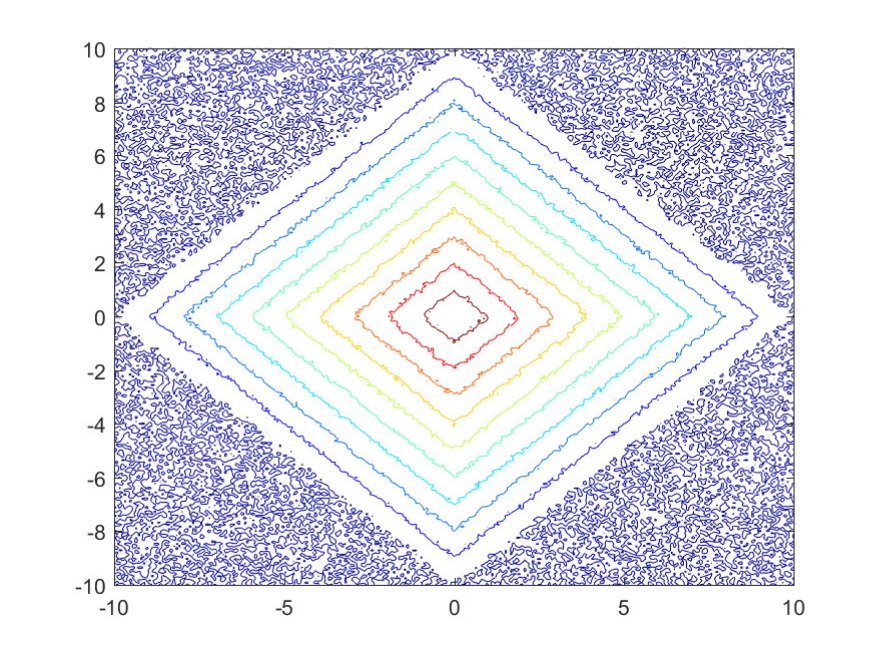}
\includegraphics[width=0.37\textwidth]{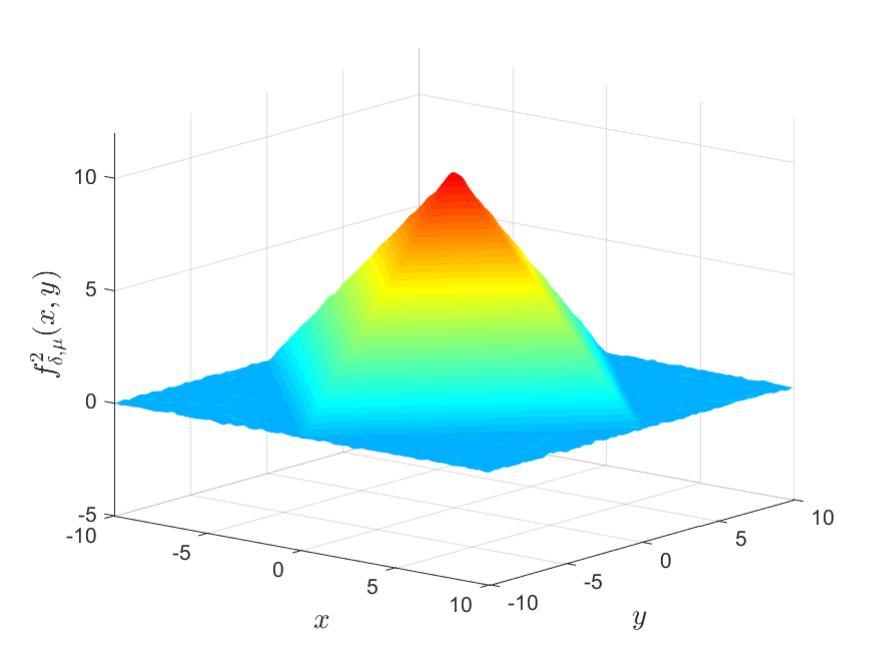}
\includegraphics[width=0.37\textwidth]{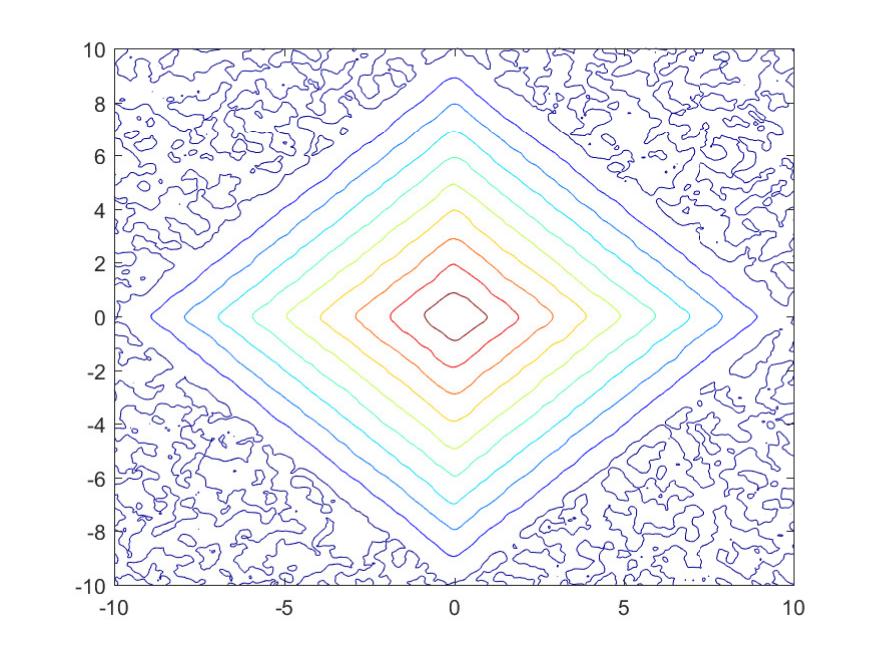}
\includegraphics[width=0.37\textwidth]{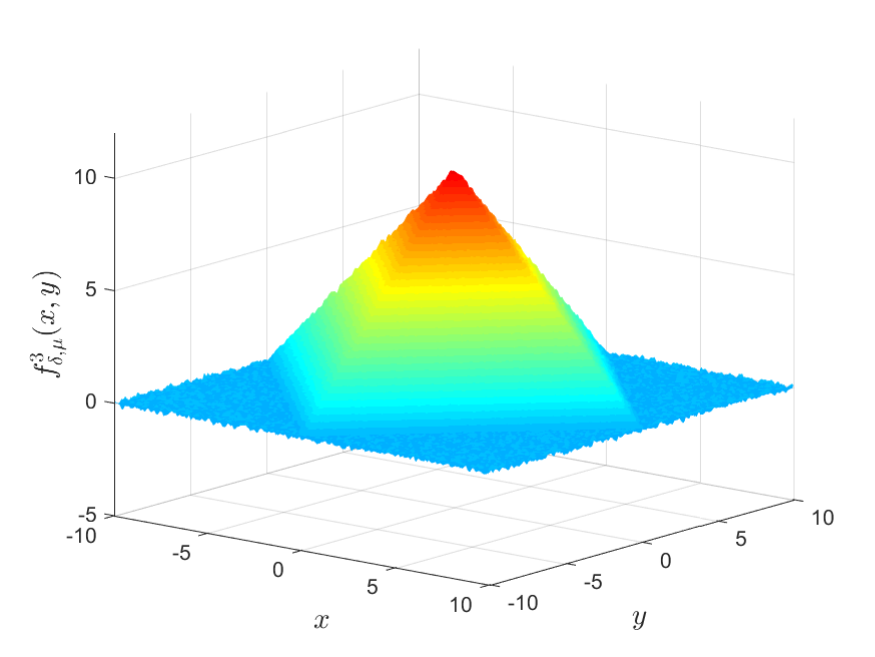}
\includegraphics[width=0.37\textwidth]{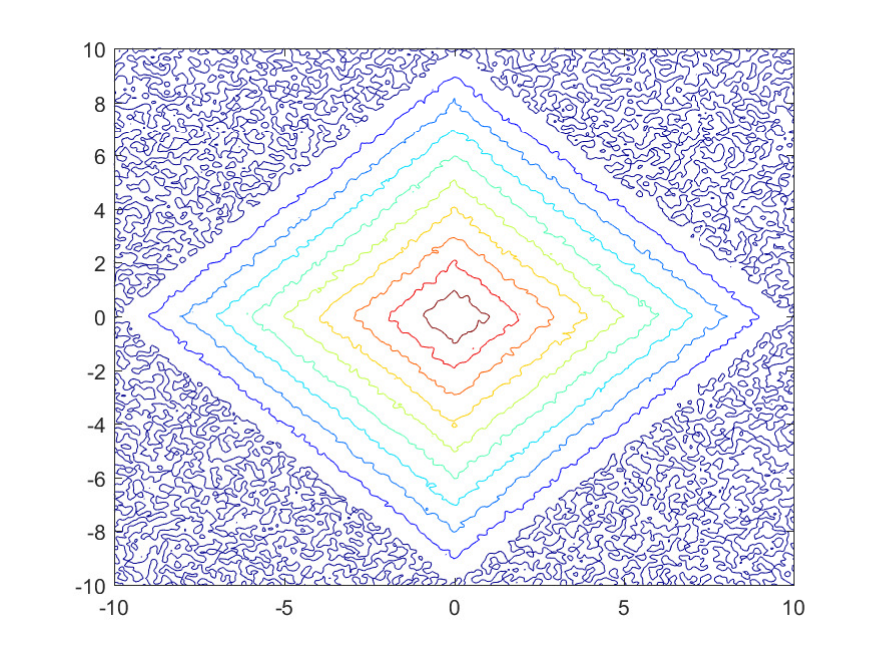}
\vspace{-0.8cm} 
\caption{Example \ref{example5}: Fonts and contour lines. Not regularized with $t_0=0.4$ (First row); regularized with $t_0=0.4$ and $p=0.6$, using $R^{1}_\mu$ (second row), $R^{2}_\mu$ (third row), $R^{3 }_\mu$ (fourth row) for $\epsilon=0.05$.}
\vspace{-0.8cm}
\label{SupPira}
\end{center}
\end{figure}

The source \eqref{fuente_ex5} retrieved in example \ref{example5} is of interest because it is a non-differentiable surface. This type of surfaces can be approximated badly in environments close to the points of non-differentiability.

As can be seen in the graphs of figure
\ref{SupPira}, once again, the regularization operators used, smooth the inverse operator making the recovery stable. No significant differences between the different operators are noticeable at first glance, although if the level curves are visualized and compared, once again, the operator $R^2_{\mu}$ provides a better estimate, more effectively stabilizing the solution of the ill-posed problem.

Table \ref{tableej5} includes the relative estimation errors that once again show the good performance obtained with the methods introduced in this paper. As observed, for this example, the operator $R^{2}_{\mu}$ presents a better general performance, since it achieves a better stabilization of the solution.
For example, for the case $\epsilon=10^{-2}$, the relative error of the estimate of the regularized source with $R^{1}_{\mu}$ is $1.21 \%$, with $R^{2}_{\mu}$ is $0.38 \%$ and with $R^{3}_{\mu}$ it is $1.65 \%$.

\begin{table}[H]
\begin{center}
{\begin{tabular}{c}\toprule
Relative errors\\
{\begin{tabular}{lccc} \toprule
\,\, $\epsilon$ &  $ \| f-f^{1}_{\delta,\mu}\|/ \left\|f\right\|$   & $ \| f-f^{2}_{\delta,\mu}\|/ \left\|f\right\|$  & $ \| f-f^{3}_{\delta,\mu}\|/ \left\|f\right\|$   \\ \midrule
$10^{-1}$  & 0.0442  & 0.0193  & 0.0491 \\
$10^{-2}$   & 0.0121  & 0.0038   & 0.0165 \\
$10^{-3}$   & 0.0056 &	0.0025  & 0.0096\\ 
$10^{-4}$   & 0.0043  & 0.0013  &  0.0079  \\
$10^{-5}$   & 0.0027  & 0.0009 & 0.0054 \\\bottomrule
\end{tabular}}
\end{tabular}}
\end{center}
\vspace{-0.30cm}
\caption{Example \ref{example5}: Relative estimation errors for $ t_0 = 1 $ and $ p = 1.$}
\vspace{-0.30cm}
\label{tableej5}
\end{table}

\subsubsection{Examples 3D}

An example of three-dimensional estimation is now addressed. Since in this case the source is a scalar function with domain in $R^3$, 4 dimensions are required to be able to plot it, because of this, only the cuts with the coordinate axes will be plotted. That is, the graphs of the retrieved sources will be analyzed when one of the variables is null.
In addition, to visualize more clearly the improvements obtained with the different regularization operators introduced and studied in this article, contour diagrams will be made for each of the cuts considered and for each regularization.

\begin{xmpl}
\label{example6}

For this example, the following parameters are considered $\alpha^2=0.4$; $\bbeta= (1,-0.5,-0.5) $; $\nu= 0.997$; $N=129 \times 129 \times 129$; $t_0=3$ and $\epsilon=0.035$. 
Finally, the source to estimate in this case is:

\begin{equation}
\label{fuente_ex6}
f(x,y,z)=\begin{cases} 
\sin\left(\dfrac{x+y+z}{20}\right), \qquad & -2\pi \leq x,y,z \leq 2\pi, \\
 0, \qquad & \text{in another case}. 
\end{cases}
\end{equation}
\end{xmpl}

The source \eqref{fuente_ex6} that is recovered in the example \ref{example6} is of interest in various problems with different characteristics due to the particularities that this function presents. It is a function, smooth, continuous and differentiable; moreover, it is symmetric with respect to the three coordinates. 

\begin{table}[h!]
\begin{center}
{\begin{tabular}{c}\toprule
Relative errors\\
{\begin{tabular}{lccc} \toprule
\,\, $\epsilon$ &  $ \| f-f^{1}_{\delta,\mu}\|/ \left\|f\right\|$   & $ \| f-f^{2}_{\delta,\mu}\|/ \left\|f\right\|$  & $ \| f-f^{3}_{\delta,\mu}\|/ \left\|f\right\|$   \\ \midrule
$10^{-1}$  & 0.0740  & 0.0365  & 0.0904 \\
$10^{-2}$   & 0.0115  & 0.0037   & 0.0127 \\
$10^{-3}$   & 0.0089 &	0.0007  & 0.0093\\ 
$10^{-4}$   & 0.0017  & 0.0006  &  0.0089  \\
$10^{-5}$   & 0.0014  & 0.0005 & 0.0053 \\\bottomrule
\end{tabular}}
\end{tabular}}
\end{center}
\vspace{-0.30cm}
\caption{Example \ref{example6}: Relative estimation errors for $ t_0 = 1 $ and $ p = 1.$}
\vspace{-0.30cm}
\label{tableej6}
\end{table}

\begin{figure}[H]
\begin{center}
\includegraphics[width=0.37\textwidth]{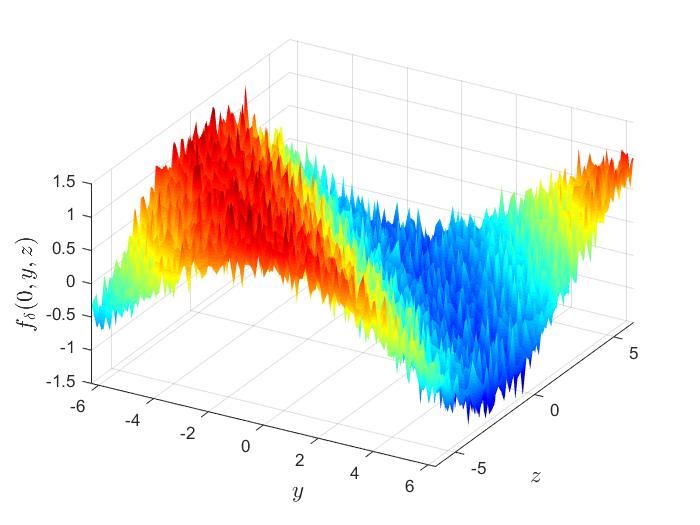}
\includegraphics[width=0.37\textwidth]{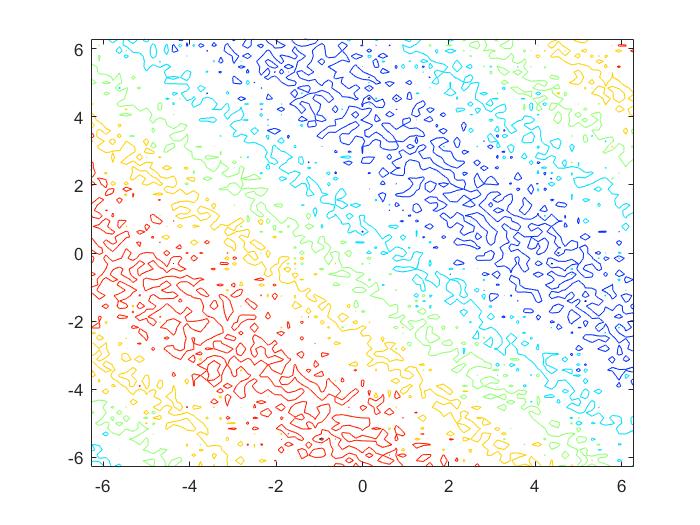}
\includegraphics[width=0.37\textwidth]{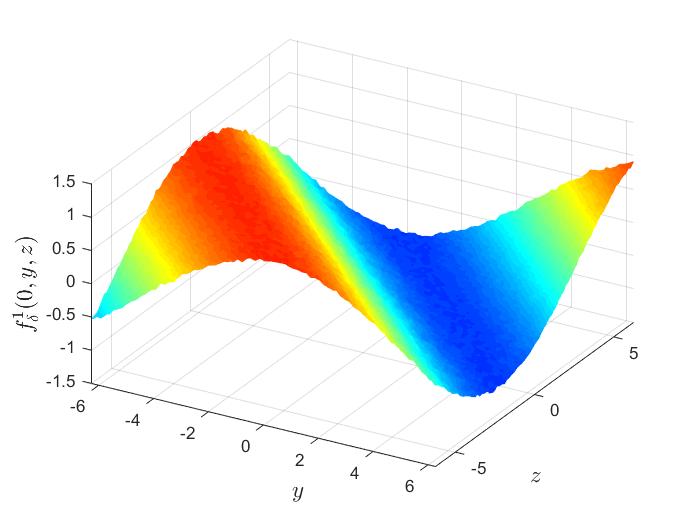}
\includegraphics[width=0.37\textwidth]{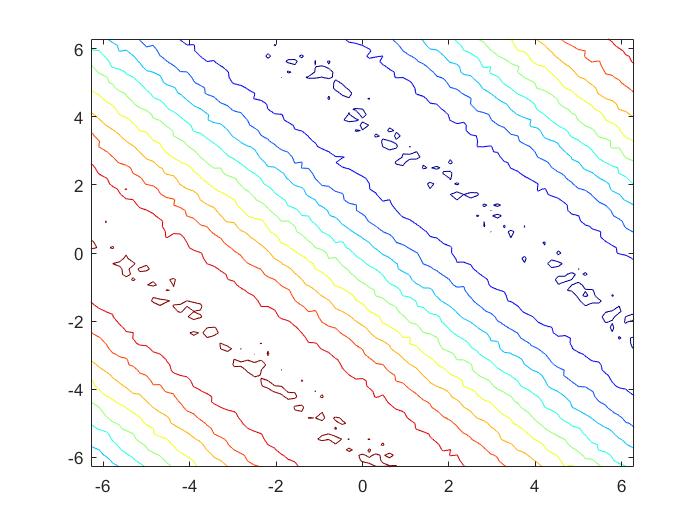}
\includegraphics[width=0.37\textwidth]{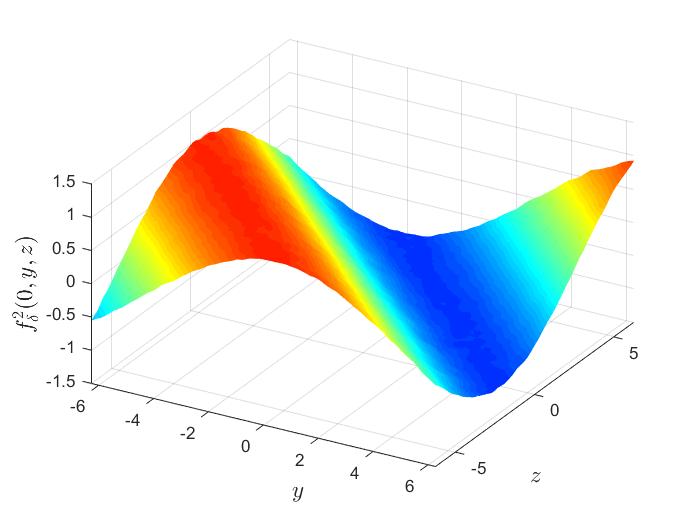}
\includegraphics[width=0.37\textwidth]{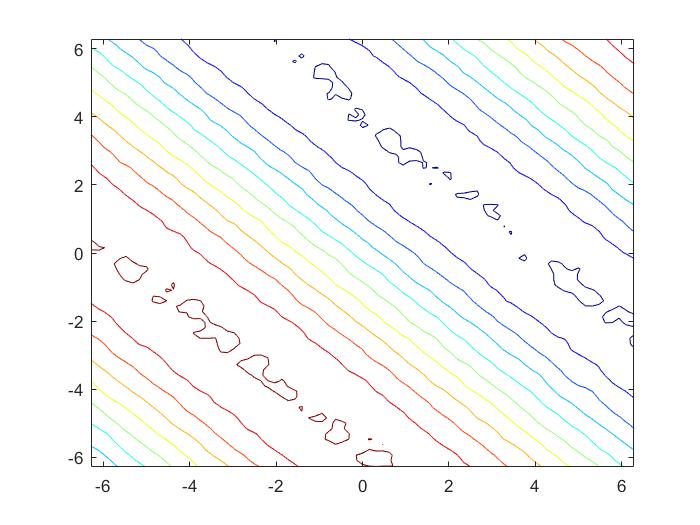}
\includegraphics[width=0.37\textwidth]{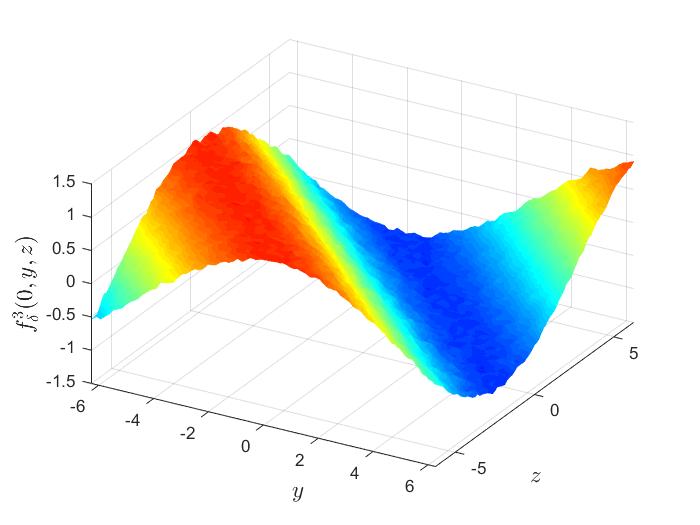}
\includegraphics[width=0.37\textwidth]{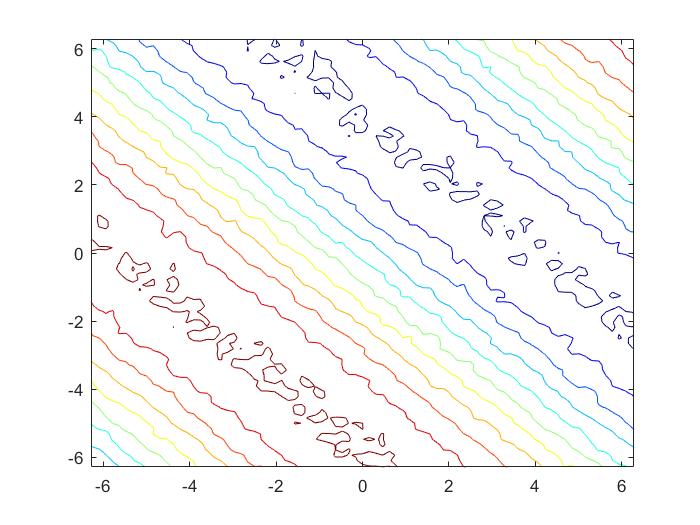}
\vspace{-0.8cm} 
\caption{Example \ref{example6}: Sources and contour lines with $(x=0)$. Not regularized with $t_0=3$ (First row); regularized with $t_0=3$ and $p=3$, using $R^{1}_\mu$ (second row), $R^{2}_\mu$ (third row), $R^{3 }_\mu$ (fourth row) for $\epsilon=0.035$.}
\vspace{0.5cm}
\label{CupSuave1}
\end{center}
\end{figure}

\vspace{-0.5cm}
\begin{figure}[H]
\begin{center}
\includegraphics[width=0.37\textwidth]{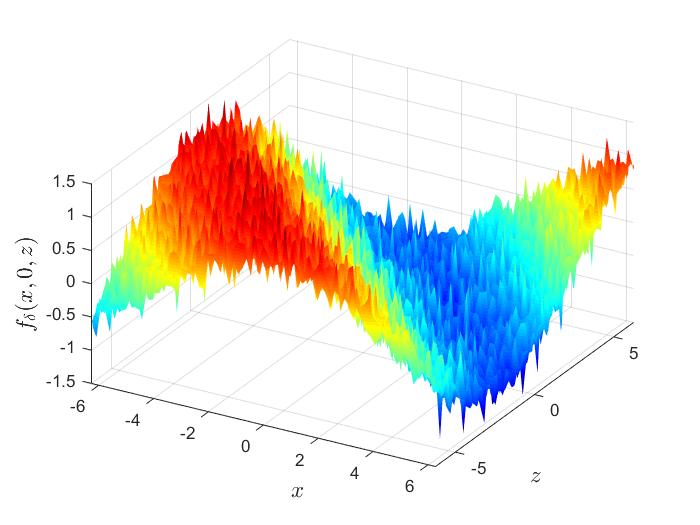}
\includegraphics[width=0.37\textwidth]{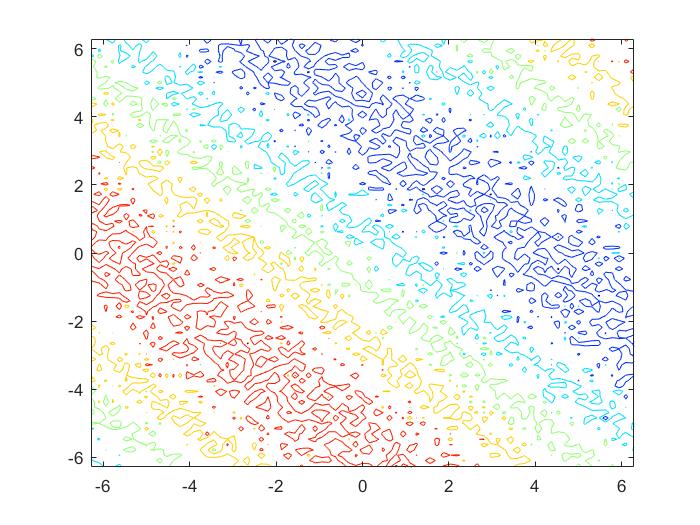}
\includegraphics[width=0.37\textwidth]{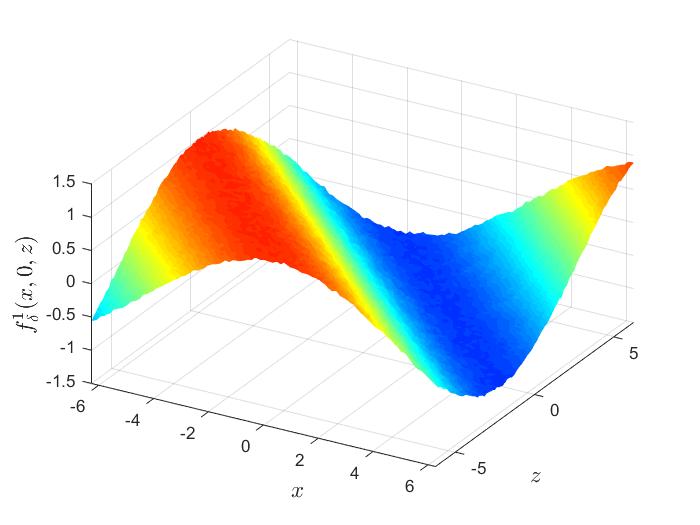}
\includegraphics[width=0.37\textwidth]{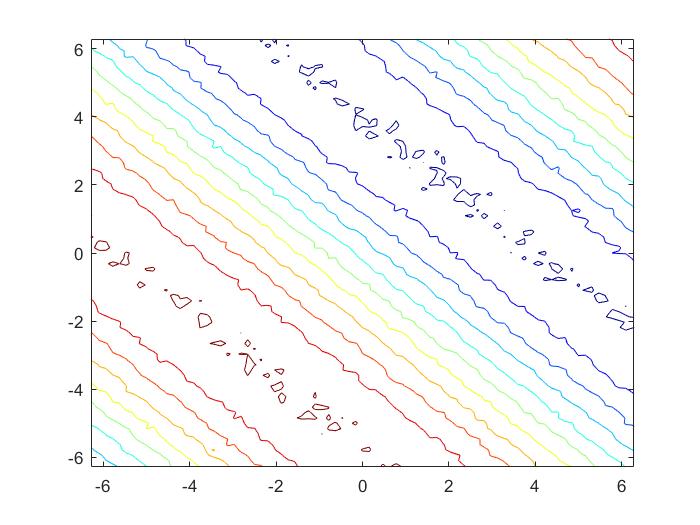}
\includegraphics[width=0.37\textwidth]{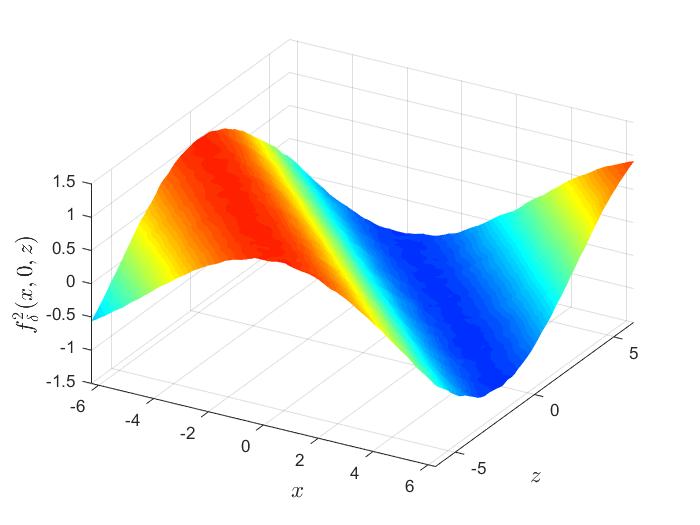}
\includegraphics[width=0.37\textwidth]{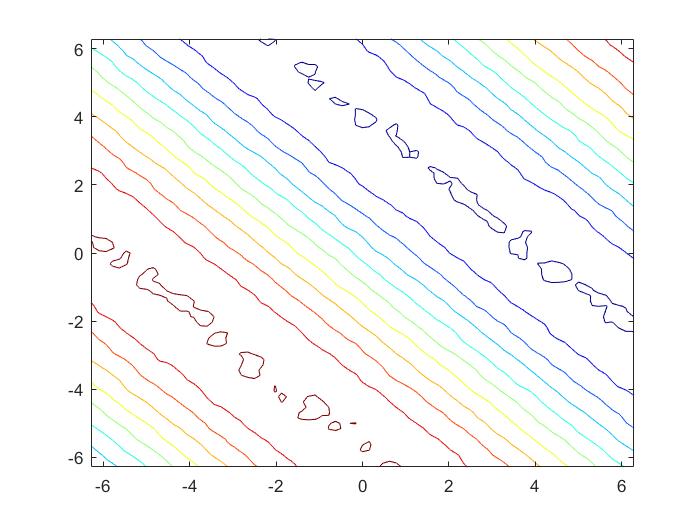}
\includegraphics[width=0.37\textwidth]{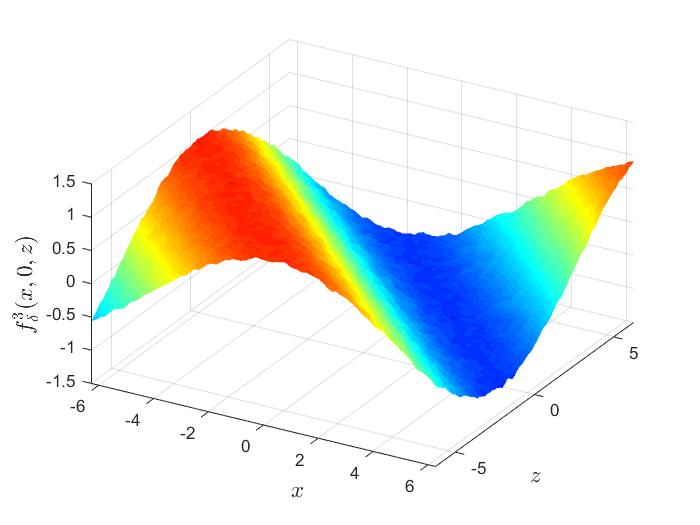}
\includegraphics[width=0.37\textwidth]{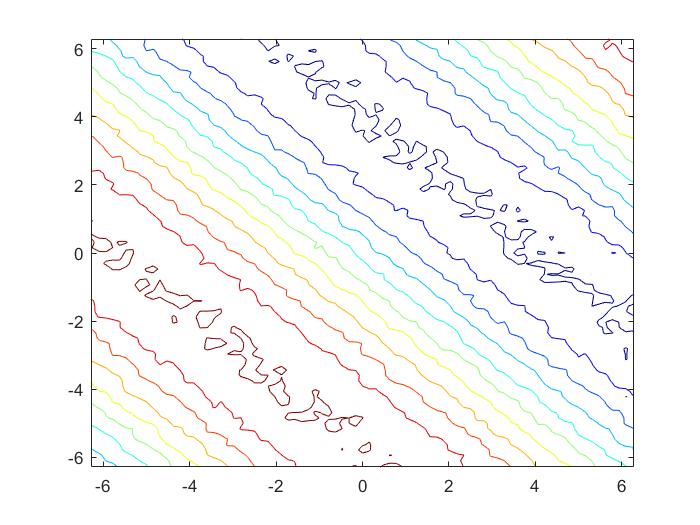}
\vspace{-0.8cm} 
\caption{Example \ref{example6}: Sources and contour lines with $(y=0)$. Not regularized with $t_0=3$ (First row); regularized with $t_0=3$ and $p=3$, using $R^{1}_\mu$ (second row), $R^{2}_\mu$ (third row), $R^{3 }_\mu$ (fourth row) for $\epsilon=0.035$.}
\vspace{1.0cm}
\label{CupSuave2}
\end{center}
\end{figure}

\begin{figure}[H]
\begin{center}
\includegraphics[width=0.36\textwidth]{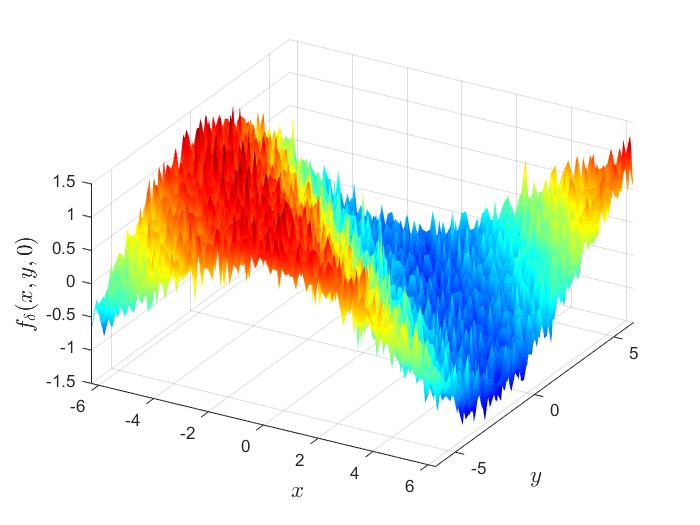}
\includegraphics[width=0.36\textwidth]{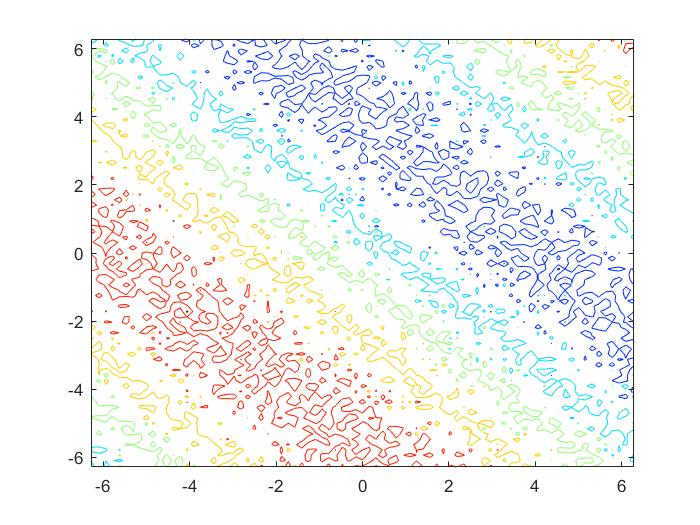}
\includegraphics[width=0.36\textwidth]{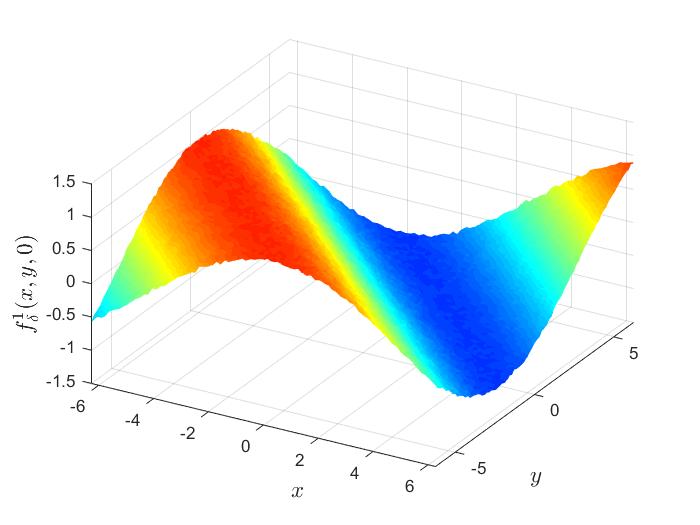}
\includegraphics[width=0.36\textwidth]{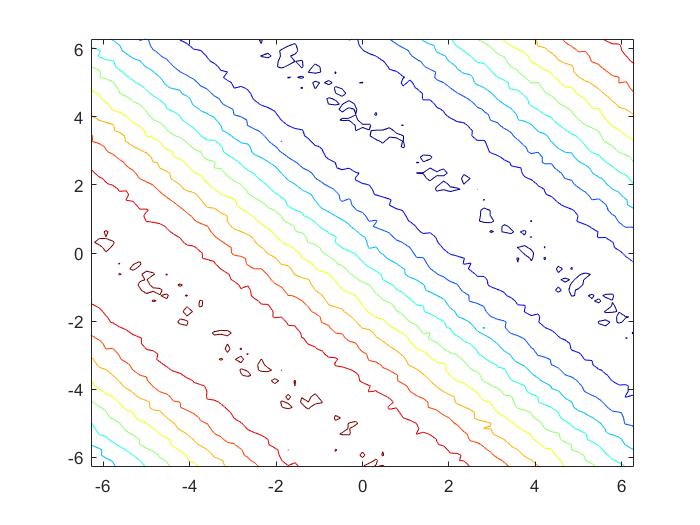}
\includegraphics[width=0.36\textwidth]{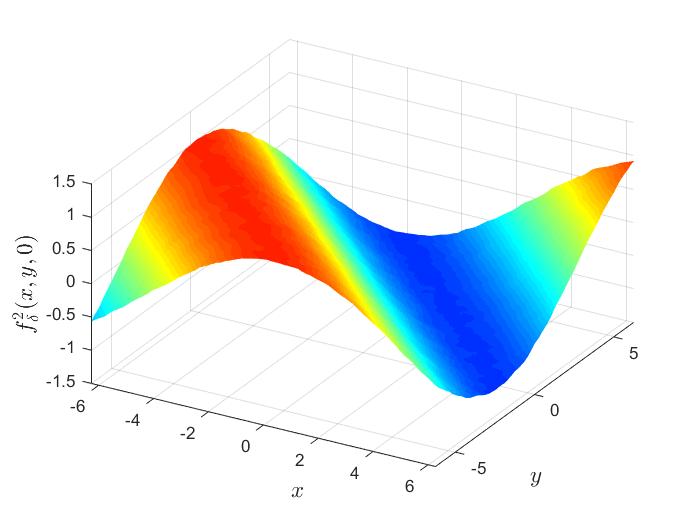}
\includegraphics[width=0.36\textwidth]{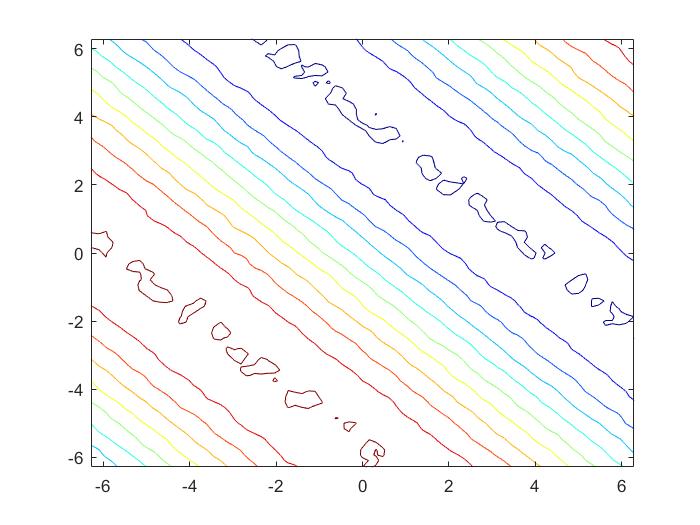}
\includegraphics[width=0.36\textwidth]{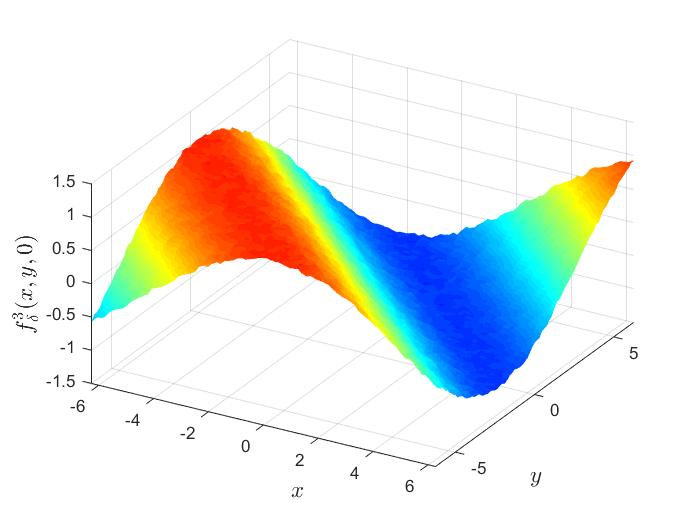}
\includegraphics[width=0.36\textwidth]{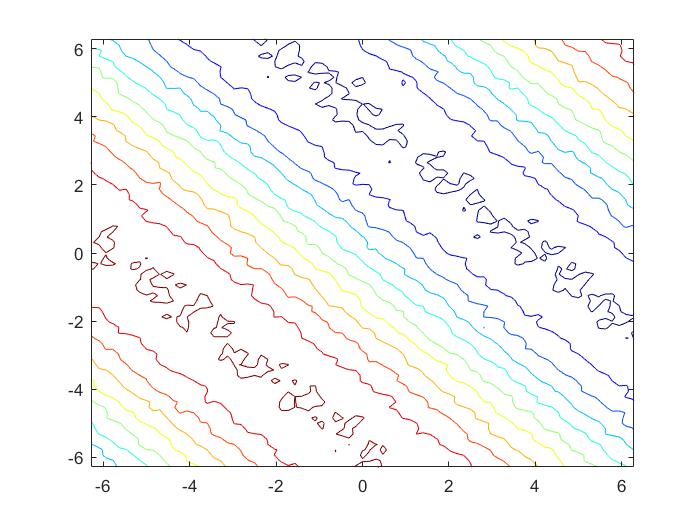}
\vspace{-0.8cm} 
\caption{Example \ref{example6}: Sources and contour lines with $(z=0)$. Not regularized with $t_0=3$ (First row); regularized with $t_0=3$ and $p=3$, using $R^{1}_\mu$ (second row), $R^{2}_\mu$ (third row), $R^{3 }_\mu$ (fourth row) for $\epsilon=0.035$.}
\vspace{1.0cm}
\label{CupSuave3}
\end{center}
\end{figure}

\vspace{1.0cm}

Table \ref{tableej6} includes the relative estimation errors that once again show the good performance obtained with the methods introduced in this article. As noted, in this example, the $R^{2}_{\mu}$ operator performs better overall.
For example, for the case $\epsilon=10^{-2}$, the relative error of the estimate of the regularized source with $R^{1}_{\mu}$ is $1.15 \%$, with $R^{2}_{\mu}$ is $0.37 \%$ and with $R^{3}_{\mu}$ it is $9.04 \%$.

As can be seen in the graphs of the figures
\ref{CupSuave1}, \ref{CupSuave2}, \ref{CupSuave3} the regularization operators used, again, smooth the inverse operator making the recovery in each cut with the coordinate axes stable.
In principle, notable differences between the recovery of each cut are not appreciated; since they are all estimated with the same order of fluctuation.
On the other hand, no significant differences between the approximations with the different regularization operators are noticeable at first glance. If the level curves are carefully observed, for each cut with the coordinate axes, it is seen that, once again, the operator $R^2_{\mu}$ yields better results.

\subsection{Discussion of results}

The numerical examples carried out show that the three regularization operators, introduced in this article, are useful to solve the problem of lack of stability in the source estimation solution. Any of the three operators studied smooth out the inverse operator making the recovery stable. This fact is independent of the general characteristics of the function to be determined, since the results are good, even for discontinuous and non-differentiable sources.

However, in general terms the operator $R^{2}_{\mu}$ yielded better results and allowed estimating the source with the lowest relative error, for the different disturbances considered. Particularly the operator $R^{3}_{\mu}$ offered a better performance when the source to be estimated is a discontinuous function.

\section{Conclusions}

This paper deals with the inverse problem of identifying the source in a complete parabolic equation
from noisy measurements. An analytical solution to the estimation problem is given and shown that the problem is ill-posed since said solution is not stable.
In order to address this drawback, three families of regularization operators specifically designed to compensate for the instability factor in the inverse operator are defined.
Besides,
a choice rule is included for the regularization parameters that is based on the level
of data noise and the smoothness of the source to be identified.
It is shown that for the proposed parameter choice rule the methods are stable. A bound is obtained for the estimation errors  that turn out to be optimal since they are of H\"older type.
 
Several numerical examples of recovering sources that belong to different spaces of Hilbert are included.
It is observed that in all of them a good performance of the adopted regularization approaches is obtained.
 On the other hand, the sources recovered by means of the inverse operator are compared with those obtained
using the regularization operators and it is concluded that the proposed regularization methods offer greater precision in the estimation of the sources considered.

\end{document}